\documentclass[preprint,12pt]{article}

\usepackage[english]{babel}	
\usepackage[utf8x]{inputenc}		
\usepackage{enumerate}
\usepackage{textcomp}                
\usepackage{typearea}
\usepackage{pdfpages}
\usepackage[toc]{appendix}
\usepackage[]{todonotes}
\usepackage{subcaption}
\usepackage{verbatim}
\usepackage{comment}
\usepackage{array}

\usepackage{amsfonts, amsmath, amssymb, amsthm, dsfont}
\usepackage{amsthm}			
\usepackage{thumbpdf}	
\usepackage{natbib}
\usepackage{booktabs}
\usepackage[noblocks]{authblk}

\usepackage{pgf, tikz}
\usepackage{tikz-cd}		
\usepackage{bbm}

\theoremstyle{definition}
\newtheorem{definition}{Definition}[]

\newtheorem*{assumption*}{Assumption}

\newtheorem*{condition*}{Condition}

\theoremstyle{plain}
\newtheorem{theorem}[]{Theorem}
\newtheorem{proposition}[definition]{Proposition}
\newtheorem{lemma}[definition]{Lemma}

\theoremstyle{remark}
\newtheorem{remark}{Remark}

\usepackage{graphicx}

\newcommand{\Z}{\mathbb{Z}}

\newcommand{\R}{\mathbb{R}}

\newcommand{\E}{\mathbb{E}}

\newcommand{\floor}[1]{\lfloor #1 \rfloor}

\newcommand{\pconv}{\xrightarrow{\mathbb{P}}}

\newcommand{\Var}{\operatorname{Var}}
\newcommand{\Cov}{\operatorname{Cov}}

\newcommand{\deq}{\overset{d}{=}}

\newcommand{\wconv}{\Rightarrow}

\newcommand{\beps}{\boldsymbol{\epsilon}}

\includecomment{proofcomplete}

\allowdisplaybreaks

\title{Rough Hurst function estimation\footnote{The majority of this work was done while both authors were employed at RWTH Aachen University, Germany, funded by the Federal Ministry of Education and Research (BMBF) and the Ministry of Culture
and Science of the German State of North Rhine-Westphalia (MKW) under the Excellence Strategy of
the Federal Government and the L\"ander, project StUpPD 413-22.}}

\author[1]{Fabian Mies}
\author[2]{Benedikt Wilkens}

\affil[1]{\footnotesize Delft University of Technology}
\affil[2]{\footnotesize RWTH Aachen University}

\begin{document}

\maketitle

\begin{abstract}
\noindent
    The fractional Brownian motion (fBm) is parameterized by the Hurst exponent $H\in(0,1)$, which determines the dependence structure and regularity of sample paths. Empirical findings suggest that the Hurst exponent may be non-constant in time, giving rise to the so-called multifractional Brownian motion (mBm). The Itô-mBm is an alternative to the classical mBm, and has been shown to admit more intuitive sample path properties in case the Hurst function is rough. In this paper, we show that the Itô-mBm also allows for a simplified statistical treatment compared to the classical mBm. In particular, estimation of the local Hurst parameter $H(t)$ with Hölder exponent $\eta>0$ achieves rates of convergence which are standard in nonparametric regression, whereas similar results for the classical mBm only hold for the smoother regime $\eta>1$. Furthermore, we derive an estimator of the integrated Hurst exponent $\int_0^t H(s)\, ds$ which achieves a parametric rate of convergence, and use it to construct goodness-of-fit tests. 
\end{abstract}

\section{Introduction}

Fractional Brownian motion (fBm) is a centered Gaussian process $B^H_t$ with covariance function $\Cov(X_s, X_t) = \frac{1}{2} (|t|^{2h} + |s|^{2H} - |t - s|^{2H})$, for a scale parameter $\sigma^2>0$ and a so-called Hurst-exponent $H\in(0,1)$. 
It admits the Mandelbrot-van Ness representation
\[
    B_t^H = \frac{1}{\Gamma(H+\frac{1}{2})}\int_{-\infty}^t \sigma\, [ (t-s)_+^{H-\frac{1}{2}} - (-s)_+^{H - \frac{1}{2}}]\, dB_s,
\]
where $B_s$ is a standard Brownian motion, and $(x)_+ = x\cdot \mathds{1}_{x>0}$.
The fBm and related processes have been widely applied in diverse fields, e.g.\ in modeling turbulent flows \citep{corcuera_asymptotic_2013,chevillard_regularized_2017} or stochastic volatility in financial markets \citep{gatheral2018volatility}. 
The Hurst exponent $H$ is of major interest, as it governs the long-range dependence of the process, the regularity of sample paths, and the self-similarity.

In some applications, e.g.\ when analyzing stock market efficiency \citep{frezza_fractal_2021}, sun-spot data \cite{Bibinger2019}, or network traffic \citep{Bianchi2004}, 
empirical evidence suggests that a single Hurst exponent $H$ is not sufficient, and a temporally-varying Hurst exponent $H_t$ should be considered. 
A nonstationary generalization of the fBm, the multifractional Brownian motion (mBm), has been introduced by \cite{Peltier1995} as 
\begin{align}
     X_t = \int_{-\infty}^t \sigma\, [ (t-s)_+^{H_t-\frac{1}{2}} - (-s)_+^{H_t - \frac{1}{2}}]\, dB_s. \label{eqn:def-classical}
\end{align}
See also \cite{Stoev2006} for related definitions of mBm.
Recently, \cite{Ayache2018} suggested an alternative nonstationary generalization of the fBm, given by
\begin{align}
     X_t = \int_{-\infty}^t \sigma_s\, [ (t-s)_+^{H_s-\frac{1}{2}} - (-s)_+^{H_s - \frac{1}{2}}]\, dB_s. \label{eqn:def-Ito}
\end{align}
Following \cite{loboda2021}, we refer to \eqref{eqn:def-classical} as the classical mBm, and to \eqref{eqn:def-Ito} as Itô-mBm. 
Both processes behave locally like a fBm, i.e.\ $h^{-H_t} (X_{t+h}-X_t) \wconv B_h^{H_t}$ as $h\to 0$. 
Thus, they are both valid candidate models for nonstationary extensions of the fBm.
Probabilistic and statistical analysis of the classical mBm is facilitated by the fact that its covariance function admits an explicit expression \citep[Thm.~4.1]{Stoev2006}. 
On the other hand, the Itô-mBm is well-defined as Itô integral if the Hurst exponent $H_t$ is an (adapted) stochastic process. 
Moreover, the latter process has been shown in \cite{loboda2021} to have attractive analytical features, see also \cite{ayache2021}.
In particular, the local Hölder coefficient at time $t$ is $\min(\eta,H_t)$ for the classical mBm \citep[Prop.~13]{Herbin2006}, and simply $H_t$ for the Itô-mBm.
The goal of this paper is to demonstrate that the Itô-mBm is not only preferable from an analytical point of view, but also allows for improved statistical inference compared to the classical mBm.

Estimators for fBm from low-frequency observations $X_1, \ldots, X_n$ are reviewed and compared by \cite{Bardet2018}. 
Optimality of the MLE has been established by \cite{Cohen2013}, and by \cite{brouste2018local} for the high-frequency regime $X_{1/n},\ldots, X_1$. 
Local nonparametric estimators for the classical mBm \eqref{eqn:def-classical} have be introduced by \cite{Coeurjolly2005}, \cite{bertrand_local_2013} and \cite{Bardet2013}.
More recently, \cite{shen2020} developed rate-optimal local nonparametric estimators of $H_t$ exploiting higher-order smoothness. 
\cite{Lebovits2017} estimate the global regularity $\min_t H_t$, and \cite{bertrand2018} describe a goodness-of-fit test for the Hurst function.
Further works on estimation of the mBm include, among others, \cite{pianese2018,garcin_estimation_2017,bianchi_modeling_2013}.
Estimation for some alternative non-stationary extensions of fBm allowing for irregular $H_t$ is studied by \cite{Benassi2000} and \cite{ayache_identification_2004,ayache_central_2007}.

The limitation of all works on the classical mBm \eqref{eqn:def-classical} is that they require $t\mapsto H_t$ to be Hölder continuous with exponent $\eta>\sup_t H_t$. 
This restriction is necessary to ensure a good local approximation of the process by a stationary fBm, see e.g.\ \citep[eqn.~7.14]{Bardet2013}. 
For the Itô-mBm, the only statistical treatment to date is due to \cite{ayache_uniformly_2023}, showing uniform consistency of a localized Hurst estimator. 
Their estimator is analogous to the estimator we propose below, but the rate derived therein is slower, indicating that their asymptotic analysis is suboptimal. 
Here, we study two estimators $\widehat{H}_n(u)$ and $\widehat{H}_n^\dagger(u)$ based on local polynomial regression, which achieve the rate $n^{-\eta/(2\eta+1)}$ for $\eta\leq 1$ (exactly for $\widehat{H}_n^\dagger$, and up to logarithmic factors for $\widehat{H}_n$).
This rate is standard in nonparametric estimation, and our results show that estimation of the Itô-mBm works \textit{just as expected} -- 
in stark contrast to estimation of the classical mBm, where estimators in the regime $\eta<1$ admit non-standard rates of convergence \citep{Bardet2013}.
In view of its elegant analytical and statistical properties, we propose the Itô-mBm as canonical nonstationary variant of the fBm. 

Beyond the local estimation of $H_t$, we derive an estimator $\widehat{\mathcal{H}}(u)$ of the integrated parameter $\mathcal{H}(u)=\int_0^u H_v\, dv$, and the error $\widehat{\mathcal{H}}-\mathcal{H}$ converges to a Gaussian process at rate $\sqrt{n}$. 
We use this estimator to construct a changepoint test for constant Hurst exponent, and goodness-of-fit test for the function $t\mapsto H_t$. 
An important feature of both tests is that they are robust to a non-constant volatility $\sigma_t$, which constitutes a nuisance under the null hypothesis. 
Our mathematical results are based on a functional central limit theorem for locally stationary time series established in \cite{mies_strong_2024}, and the proof technique linking methods for stochastic processes in discrete and continuous time might be of independent interest.

The pointwise nonparametric estimator of the Hurst exponent is studied in Section \ref{sec:local-est}. 
The integrated parameter estimator is introduced in Section \ref{sec:integrated}, including the changepoint test and the goodness-of-fit test.
All technical proofs are gathered in the Appendix.

\section{Local nonparametric estimation}\label{sec:local-est}

Performing high-frequency inference for the the Itô-mBm \eqref{eqn:def-Ito}, we want to estimate the exponent $H_s$ nonparametrically, using discrete observations $X_{\frac{1}{n}}, \ldots, X_{\frac{n}{n}}$. 
For simplicity of exposition, we suppose that $X_{-\frac{k}{n}}$ for $k=1,2,3$ are also observable, and we define the second order increments as
\begin{align*}
    \chi_{i,n} &= X_{\frac{i}{n}} - 2 X_{\frac{i-1}{n}} + X_{\frac{i-2}{n}}, \qquad i=1,\ldots, n, \\
    \widetilde{\chi}_{i,n} &= X_{\frac{i}{n}} - 2 X_{\frac{i-2}{n}} + X_{\frac{i-4}{n}}, \qquad i=1,\ldots, n.
\end{align*}
Our estimation procedure is based on the change-of-frequency principle introduced by \cite{KentWood1997}, which is a special case of the quadratic variation estimator of \cite{istas1997}, and has also been used by \cite{coeurjolly_estimating_2001}: If $H_s=H$ and $\sigma_s=\sigma$ are constant, the self-similarity of the fBm yields $\E(\widetilde{\chi}_{i,n}^2)/ \E(\chi_{i,n}^2) = 2^{2H}$, such that the moment estimator given by $\widehat{H}=\frac{1}{2}\log_2 \left(\sum_i \widetilde{\chi}_{i,n}^2 /\sum_i {\chi}_{i,n}^2\right)$ is consistent at rate $\sqrt{n}$.
For estimation of the Itô-mBm, we use the same idea, but localize the moment estimator around time $u\in(0,1)$. 
To this end, we use a kernel function $K:\R\to[0,\infty)$ satisfying
\begin{align}
    \begin{split}
        \int K(x)\, dx =1, \qquad K(x)=0 \text{ for } |x|>1, \\
    \{x: K(x)>0\} \text{ has positive Lebesgue measure}.
    \end{split}\label{ass:kernel} \tag{K}
\end{align}and a bandwidth $b=b_n>0$, denote $K_b(x)=\frac{1}{b} K(\frac{x}{b})$.
For a bandwidth $b=b_n>0$, denote $K_b(x) = \tfrac{1}{b} K(\tfrac{x}{b})$, and for any $u\in(0,1)$, consider the locally weighted polynomial regression \citep{fan_local_2018,Tsybakov2008}
\begin{align*}
    \min_q \sum_{i=1}^n K_{b_n}(\tfrac{i}{n}-u) \left\| \begin{pmatrix}
        \chi_{i,n}^2 \\ \widetilde{\chi}_{i,n}^2
    \end{pmatrix}  - q(\tfrac{i}{n}) \right\|^2,
\end{align*}
where the minimum $q^*$ is determined among all polynomials $q:(0,1)\to \R^2$ of degree $l$
We use $\widehat{\phi}_n(u)=q^*(u)$ as moment estimator at location $u$.
It is well known \citep[Lem.~1.3~\&~Prop.~1.12]{Tsybakov2008} that the estimator admits the linear representation
\begin{align*}
    \widehat{\phi}_n(u) = \sum_{i=1}^n w_{i,n}(u) \begin{pmatrix}
        \chi_{i,n}^2 \\ \widetilde{\chi}_{i,n}^2
    \end{pmatrix},
\end{align*}
and the weights satisfy, for some universal $C>0$, and for all $u\in [b_n, 1-b_n]$,
\begin{enumerate}[(i)]
    \item $\sup_{i} |w_{i,n}(u)| \leq \frac{C}{nb_n}$,
    \item $\sum_{i=1}^n |w_{i,n}(u)| \leq C$,
    \item $w_{i,n}(u)=0$ for $|\frac{i}{n}-u| > b_n$
    \item $\sum_{i=1}^n w_{i,n}(u)=1$ and $\sum_{i=1}^n (\frac{i}{n}-u)^k w_{i,n}(u)=0$ for $k=1,
\ldots,l$.
\end{enumerate}

\begin{remark}\label{rem:uniform}
    If $K$ is supported on $[-1,0]$ such that still $\{x\in[-1,0]:K(x)>0\}$ has positive Lebesgue measure, then the latter properties (i)--(iv) hold for all $u\in[b_n, 1]$.
    Similarly, if $K$ is supported on $[0,1]$, then (i)--(iv) hold for all $u\in [0,1-b_n]$.
    Moreover, if $K(x)>K_0 \mathds{1}(|x|<\tau)$ for some small $\tau>0$ and $K_0>0$, then properties (i)--(iv) hold for all $u\in [0,1]$.
\end{remark}

We estimate $H_u$ via the log-ratio estimator
\begin{align*}
    \widehat{H}_n(u) 
    &= \left(\frac{1}{2} \log_2 \left(\frac{\widehat{\phi}_{n,2}(u)}{\widehat{\phi}_{n,1}(u)} \right) \right)\;\vee 0 \; \wedge 1.
\end{align*}

\begin{theorem}\label{thm:rate}
    Suppose that $v\mapsto \theta_v = (\sigma_v, H_v)$ is Hölder continuous with exponent $\eta\in(0,l+1]$, and that $0<\underline{H} \leq H_v \leq \overline{H} < 1$ for all $v$, and $0<\underline{\sigma}^2\leq \sigma_v^2 \leq \overline{\sigma}^2<\infty$.
    If $b_n \ll n^{q-1}$ for some $q\in(0,1)$, then for any $p\geq 2$,
    \begin{align*}
        \widehat{H}_n(u) 
        &= H_u + \mathcal{O}_{L_p}\left( \log(n)^{\lceil \eta\rceil} b_n^{\eta} + \frac{1}{\sqrt{n\, b_n}} \right) + \mathcal{O}\left( \log(n) n^{-(1\wedge \eta)} \right). 
    \end{align*}
    The bound holds uniformly for $\theta \in [\underline{\sigma}^2,\overline{\sigma}^2] \times [\underline{H},\overline{H}]$, and $u\in[b_n,1-b_n]$.
    The upper bound is minimized for $b_n \asymp n^{-\frac{1}{2\eta+1}} \log(n)^{-\frac{2\lceil \eta\rceil}{2\eta+1}}$, so that
    \begin{align*}
        \widehat{H}_n(u) 
        &= H_u + \mathcal{O}_{L_p}\left( \log(n)^{\frac{\lceil \eta\rceil}{2\eta+1}} n^{-\frac{\eta}{2\eta+1}} \right).
    \end{align*}
\end{theorem}

\begin{remark}
    Under the conditions described in Remark \ref{rem:uniform}, the bounds of Theorem \ref{thm:rate} and Proposition \ref{prop:NW} in the appendix can be extended to $u\in[b_n, 1]$, $u\in[0, 1-b_n]$, and $u\in [0,1]$, respectively.
\end{remark}

For the locally constant variant, i.e.\ $l=1$, this estimator is the same as the one considered by \cite{Coeurjolly2005} for the classical mBm.
Central to this approach is the scaling relation $\E(\widetilde{\chi}_{i,n}^2)/ \E(\chi_{i,n}^2) \approx 2^{2H_{i/n}}$, and the estimator $\widehat{H}(u)$ utilizes this by estimating the mean first, and then taking the ratio. 
An alternative approach pursued by \cite{shen2020} is to take the ratio first, and smooth second.
This leads to the local polynomial estimator
\begin{align*}
    \widehat{H}^\dagger_n(u) = \frac{1}{2}\sum_{i=1}^n w_{i,n}(u) \log_2\left(\frac{ \widetilde{\chi}_{i,n}^2}{\chi_{i,n}^2}\right).
\end{align*}
Again, we may constrain the estimator manually to the interval $[0,1]$. 
It turns out that this estimator leads to an improvement of the rate of convergence by a logarithmic factor. 
Compared to the results of \cite{shen2020} for the classical mBm, we show that this estimation approach, when applied to the Itô-mBm, also works for less smooth Hurst functions with Hölder exponent $\eta<1$.

\begin{theorem}\label{thm:rate2}
    Under the conditions of Theorem \ref{thm:rate}, the estimator $\widehat{H}^\dagger_n$ satisfies
    \begin{align*}
        \widehat{H}^\dagger_n(u) 
        &= H_u + \mathcal{O}\left( b_n^\eta \right) + \mathcal{O}_{L_2}\left( \frac{1}{\sqrt{n\, b_n}} \right) + \mathcal{O}\left( \log(n) n^{-(1\wedge \eta)} \right). 
    \end{align*} 
    The upper bound is minimized for $b_n \asymp n^{-\frac{1}{2\eta+1}}$, so that
    \begin{align*}
        \widehat{H}^\dagger_n(u) 
        &= H_u + \mathcal{O}_{L_2}\left( n^{-\frac{\eta}{2\eta+1}} \right).
    \end{align*}
\end{theorem}

\begin{remark}
    The estimator of \cite{shen2020} achieves the rate $\mathcal{O}((n\log (n)^2)^{-\frac{\eta}{2\eta+1}})$. 
    The logarithmic advantage can be attributed to the fact that they suppose $\sigma$ to be constant in time.
    It may thus be estimated at a faster rate, and is essentially known for the local estimation of $H_u$, which allows for slightly better estimation of the latter. The same phenomenon of logarithmically faster rates if $\sigma$ is known has been shown for the stationary case of the fBm \citep{brouste2018local}. 
\end{remark}
\begin{remark}
    The slower rate of the estimator $\widehat{H}_u$ is due to a logarithmically larger bias of the estimator $\widehat{\phi}_n(u)$. 
    In particular, Lemma \ref{lem:bias} in the appendix shows that $\E \chi_{i,n}^2 \approx n^{2H_{i/n}} \sigma^2_{i/n} \Gamma_{H_{i/n}}(0)$, and since $H$ appears in the exponent, differentiating with respect to $H$ to control the bias yields additional terms of order $\log(n)$.
\end{remark}

\section{Integrated parameter estimation}\label{sec:integrated}

Another approach to treat the Hurst parameter nonparametrically is via the integrated parameter 
\begin{align*}
    \mathcal{H}(u) = \int_0^u H_v\, dv.
\end{align*}
Interest in this functional arises because many hypotheses about $H_u$ may be formulated in terms of $\mathcal{H}(u)$. 
For instance, the no-change hypothesis $H_v=H_0$ for all $v$ is equivalent to $\mathcal{H}(u)=u\mathcal{H}(1)$; monotonicity of $H_u$ is equivalent to convexity of $\mathcal{H}(u)$; and $H_v\geq \underline{H}$ is equivalent to $\mathcal{H}(u_2)-\mathcal{H}(u_1) \geq \underline{H}(u_2-u_1)$ for all $u_2\geq u_1$.
We will discuss these examples in detail below.

A naive estimator for the integrated Hurst parameter $\mathcal{H}(u)$ can be obtained by integrating the local estimator, i.e.\ $\widetilde{\mathcal{H}}(u) = \int_0^u \widehat{H}_n(v)\, dv$.
However, as discussed in \citep{mies2021}, this estimator will in general be asymptotically bias-driven.
Instead, we employ a generic linearization procedure which has been introduced in \citep{mies2021} in the context of locally-stationary time series.
To this end, we propose the estimator
\begin{align*}
    \widehat{\mathcal{H}}(u) = \frac{1}{n} \sum_{t=2L}^{\lfloor un\rfloor } \left\{ \widehat{H}_n(\tfrac{t-L}{n}) +  \frac{\widetilde{\chi}_{t,n}}{2\widehat{\phi}_{n,2}(\frac{t-L}{n})} -
        \frac{\chi_{t,n}}{2\widehat{\phi}_{n,1}(\frac{t-L}{n})}   \right\},
\end{align*}
where $\widehat{H}_n$ uses a kernel which is supported on $[-1,0]$. 
That is, $\widehat{H}_n(\frac{t}{n})$ only uses data up to time $\frac{t}{n}$.
\begin{remark}\label{rem:boundary}
    If $t\leq \lceil n b_n\rceil$, the one-sided estimators $\hat{\phi}_n(\frac{t}{n})$ and $\widehat{H}_n(\frac{t}{n})$ effectively use the bandwidth $\frac{t}{n} \leq b_n$.
    This is due to the boundary adaptation of the local polynomial regression.
\end{remark}

\begin{theorem}\label{thm:integrated-clt}
    Suppose that $v\mapsto \theta_v = (\sigma_v, H_v)$ is Hölder continuous with exponent $\eta\in(\frac{1}{2},1]$, and that $0<\underline{H} \leq H_v \leq \overline{H} < 1$ for all $v$, and $0<\underline{\sigma}^2\leq \sigma_v^2 \leq \overline{\sigma}^2<\infty$.
    Suppose furthermore that $b_n$, $L_n$ are chosen such that
    \begin{align*}
        n^{-\frac{1}{2}+r} \ll b_n &\ll n^{-\frac{1+2r}{4\eta}} \quad \text{ for some $r\in(0,\tfrac{1}{2})$},\\
        n^{\frac{1}{2} - r} \log(n)^2 \ll L_n &\ll n^{\frac{1}{2}-\delta} \quad \text{ for some $\delta\in(0,\tfrac{1}{2})$}.
    \end{align*}
    Then, as $n\to\infty$
    \begin{align*}
        \sqrt{n} (\widehat{\mathcal{H}}(u) - \mathcal{H}(u)) \quad\wconv\quad \int_0^u \tau(H_v)\, dW_v \overset{d}{=} W\left( \int_0^u \tau^2(H_v)\, dv \right),
    \end{align*}
    in the Skorokhod space $D[0,1]$. 
    The local asymptotic variance $\tau^2(H)$ is defined in \eqref{eqn:asymp-var}.
\end{theorem}
\begin{remark}
    The condition on $L_n$ is rather weak, and any choice $L_n \asymp n^{a}$ for some $a\in(\frac{1}{6}, \frac{1}{2})$ will satisfy the conditions. 
    The choice of $b_n$ basically needs to ensure that the local estimator $\widehat{\phi}_n(u)$ and hence $\widehat{H}_n(u)$, admits a sufficiently fast rate of convergence of order $\mathcal{O}(n^{-\frac{1}{4} - \frac{r}{2}})$. 
    The feasible range of choices for $b_n$ allows for this rate to be driven by bias or variance, and may always be satisfied for $\eta>\frac{1}{2}$.
\end{remark}

We denote the limiting variance process by $\Sigma(u) = \int_0^u \tau^2(H_v)\, dv$. 
To perform statistical inference, we may estimate this process via the plug-in method as $\widehat{\Sigma}(u)=\frac{1}{n}\sum_{t=2L_n}^{\lfloor un\rfloor} \tau^2(\widehat{H}_n(\tfrac{t}{n}))$.

\begin{theorem}\label{thm:studentizing}
    Under the conditions of Theorem \ref{thm:integrated-clt}, $\sup_{u\in[0,1]} \left|\Sigma(u)-\widehat{\Sigma}(u)\right| \pconv 0$.
\end{theorem}
That is, the limiting process $W(\Sigma(u))$ of Theorem \ref{thm:integrated-clt} is approximated by $W(\widehat{\Sigma}(u))$, which allows for feasible statistical inference.
In the next two subsections, we proceed to describe two specific hypothesis tests based on the estimator for the integrated Hurst parameter.

\subsection{Testing for constant Hurst exponent}

The estimator for the integrated Hurst exponent can be used to test for a constancy, that is, to treat the statistical problem
\begin{align*}
    \mathbb{H}_0: H_v \text{ constant} \quad\leftrightarrow\quad \mathbb{H}_1: H_v\text{ not constant}.
\end{align*}
Rejecting $\mathbb{H}_0$ is interpreted as evidence that a model based on fBm is not sufficient, and multifractional extensions need to be considered.
This problem has been studied by \cite{Bibinger2019}. 
Therein, the specific multifractional model, i.e.\ Itô-mBm \eqref{eqn:def-Ito} vs classical mBm \eqref{eqn:def-classical}, does not matter, because the process is stationary under the null. 
\cite{Bibinger2019} employs a CUSUM statistic based on the squared increments $\chi_{t,n}^2$. The test is applied to sunspot data, finding evidence for nonstationarity.
A methodological limitation of the referenced method is that it will detect both changes in $\sigma_v$ and in $H_v$, and a post-hoc analysis is necessary to distinguish both types of changes. 
In contrast, we would like to design a test which is only sensitive to changes in $H_v$, while being robust against changes in $\sigma_v$.
That is, we treat the volatility as a nuisance. 
The relevance of allowing for nonstationary nuisance quantities under the no-change null hypothesis was first recognized by \cite{Zhou2013}, and further developed by \citet{Dette2018,Gorecki2018,Pesta2018,Demetrescu2018,schmidt_asymptotic_2021,cui_estimation_2021}. 
To the best of our knowledge, changepoint testing with nonstationary nuisance quantities has not yet been considered for continuous-time processes.

To test for a change, we suggest the CUSUM-type statistic
\begin{align*}
    T_{\mathrm{CUSUM}}(\widehat{\mathcal{H}}) = \sup_{u\in[0,1]} \left| \widehat{\mathcal{H}}(u)-u\widehat{\mathcal{H}}(1) \right|.
\end{align*}
Under $\mathbb{H}_0$, and if the volatility function satisfies the conditions of Theorem \ref{thm:integrated-clt}, the statistic $\sqrt{n} T(\widehat{\mathcal{H}})$ will converge in distribution to $\sup_{u\in[0,1]} |W(\Sigma(u))-uW(\Sigma(1))|$.
This can be used to derive critical values.

\begin{proposition}\label{prop:CUSUM}
    Suppose that the conditions of Theorem \ref{thm:integrated-clt} hold.
    Let $q_n(\alpha)$ be the $1-\alpha$, $X$-conditional quantile of the random variable $Y_n=\sup_{u\in[0,1]} \left| W\left(\widehat{\Sigma}(u)\right) - u W(\widehat{\Sigma}(1))\right|$.
    If $H_v$ is constant, then 
    \begin{align*}
        \limsup_{n\to\infty} P\left(\sqrt{n}\,T_{\mathrm{CUSUM}}(\widehat{\mathcal{H}})> q_n(\alpha)\right) \leq \alpha.
    \end{align*}
    Alternatively, if $H_v$ is not constant, then
    \begin{align*}
        \lim_{n\to\infty} P\left(\sqrt{n}\,T_{\mathrm{CUSUM}}(\widehat{\mathcal{H}})> q_n(\alpha)\right) = 1.
    \end{align*}
\end{proposition}

\subsection{Application to goodness-of-fit testing}

Let $\mathcal{G}_0\subset \{ H:[0,1]\to (0,1) \}$ be a class of candidate functions for the Hurst parameter, and we want to test the null hypothesis 
\begin{align*}
    \mathbb{H}_0: H\in \mathcal{G}_0 \quad \leftrightarrow\quad \mathbb{H}_1: H\notin\mathcal{G}_0.
\end{align*}
For example, setting $\mathcal{G}_0=\{H_v = av+b \,|\, b\in(0,1), a+b\in(0,1) \}$ yields a test for linearity, and setting $\mathcal{G}_0 = \{ H\,|\, \exists v_0 \text{ s.t.\ $H$ is increasing on $[0,v]$ and decreasing on $[v,1]$} \}$ tests for unimodality.
We suggest to apply the test statistic
\begin{align*}
    \widehat{T}(\mathcal{G}_0) = \inf_{\tilde{H}\in \mathcal{G}_0} \widehat{T}(\tilde{H}), \quad\text{where }
    \widehat{T}(\tilde{H})=\sup_{u\in [0,1]} \left| \widehat{\mathcal{H}}(u) - \int_0^u \tilde{H}(v)\, dv \right|.
\end{align*}
Under $\mathbb{H}_0$, we clearly have $\sqrt{n}\widehat{T}(\mathcal{G}_0)\leq \sqrt{n}\widehat{T}(H) \wconv \sup_{u\in[0,1]} \left|\int_0^u \tau(H_v)\, dW_v \right|$, and we can use the quantiles of the latter as critical values.
Importantly, we can estimate the asymptotic variance function as in Theorem \ref{thm:studentizing}, which is also consistent under the alternative.
As a result, we obtain a consistent test.

\begin{proposition}\label{prop:GOF}
    Suppose that the conditions of Theorem \ref{thm:integrated-clt} hold.
    Let $q_n(\alpha)$ be the $1-\alpha$, $X$-conditional quantile of the random variable $Z_n=\sup_{u\in[0,1]} | W(\widehat{\Sigma}(u) )|$.
    If $H\in \mathcal{G}_0$, then 
    \begin{align*}
        \limsup_{n\to\infty} P\left(\sqrt{n}\,\widehat{T}(\mathcal{G}_0)> q_n(\alpha)\right) \leq \alpha.
    \end{align*}
    Alternatively, if $H\notin \mathcal{G}_0$ and $\mathcal{G}_0$ is closed with respect to the uniform norm, then
    \begin{align*}
        \lim_{n\to\infty} P\left(\sqrt{n}\,\widehat{T}(\mathcal{G}_0)> q_n(\alpha)\right) = 1.
    \end{align*}
\end{proposition}

An alternative goodness-of-fit test for the classical mBm and a singleton null $\mathcal{G}_0 = \{\tilde{H}\}$ has been suggested by \cite{bertrand2018}, based on a Cramer-von-Mises type test statistic of the form $\sum_{l=1}^L |\widehat{H}_n(\frac{l}{l}) - \tilde{H}(\frac{l}{L})|^2$ which they show to be asymptotically normal if $L=L_n\to \infty$ suitably as $n\to\infty$. However, they do not specify the constraints on $L_n$, making their test practically infeasible. 
Moreover, they suppose $\sigma$ to be constant.

\section{Discussion}

We constructed two local nonparametric estimators for the Hurst exponent of the Itô-mBm which achieve standard nonparametric rates of convergence even in the smoothness regime $\eta<1$. 
In contrast, existing estimators for the classical mBm require $\eta>1$. 
This suggests that the statistical treatment of the Itô-mBm is in some sense simpler, as the smoothness of the Hurst function is less critical, which is in analogy to previous findings on its path regularity. 
However, it is currently not clear whether the difficulties of estimating the classical mBm with a roguher Hurst function are intrinsic to the problem, or could be overcome by improved estimation procedures. Future work might further explore this by either deriving statistical lower bounds for estimation of the classical mBm, or by achieving standard rates of convergence with alternative estimators.

Although we argue that the Itô-mBm is superior to the classical mBm for statistical modeling, a big advantage of the latter is its explicit covariance function, which enables exact likelihood inference and efficient simulation.
Indeed, we attempted to simulate the Itô-mBm by discretizing the stochastic integral, but could not reach satisfactory accuracy for a statistical simulation study.

Lastly, we advocate for the use of Itô-mBm over the classical mBm on the basis of its attractive analytical and statistical properties. 
However, this is a purely mathematical argument and not based on empirical evidence. 
Obviously, the looming open questions are: (i) Can we distinguish the two models based on data?
(ii) Do the two models lead to different conclusions for practical questions, e.g.\ forecasting or asset pricing? 
We formulate these open questions as promising directions for future research, especially in view of the recent success of fractional models for stochastic volatility \citep{gatheral2018volatility,chong_nonparametric_2025}.

\appendix
\section{Multiplier invariance principle}\label{sec:multiplier}

For the proof of Theorem \ref{thm:integrated-clt}, we make use of a functional central limit theorem for nonstationary time series, developed in prior work \citep{mies_strong_2024}. To make this article self-contained, we repeat the essential assumptions and the result in this section. 

For iid random seeds $\epsilon_i \sim U(0,1)$, and functions $G_{t,n}:\R^\infty \to \R^d$, $t=1,\ldots, n$, define the nonstationary array of time series $X_{t,n}$ as 
\begin{align*}
    X_{t,n} &= G_{t,n}(\beps_t), \quad t=1,\ldots, n, \\
    \beps_t &= (\epsilon_t, \epsilon_{t-1},\ldots) \in\R^\infty.
\end{align*}
Throughout this section, we assume that $\E(X_{t,n})=0$.
For an independent copy $\tilde{\epsilon}_i\sim U(0,1)$, define also 
\begin{align*}
    \tilde{\beps}_{t,h} &= (\epsilon_t, \ldots, \epsilon_{t-h+1}, \tilde{\epsilon}_{t-h},\epsilon_{t-h-1}\ldots) \in \R^\infty, \\
    \beps_{t,h} &= (\epsilon_t, \ldots, \epsilon_{t-h+1}, \tilde{\epsilon}_{t-h},\tilde{\epsilon}_{t-h-1}\ldots) \in \R^\infty.
\end{align*}
We first impose the following set of assumptions, for some $\Gamma_n\geq 1$, $q>2$, and $\beta>1$:
\begin{align*}
    \|G_{1,n}(\beps_0)\|_{L_2} + \sum_{t=2}^n \|G_{t,n}(\beps_0) - G_{t-1,n}(\beps_0)\|_{L_2}  &\leq \Theta_n \Gamma_n, \tag{A.1} \label{ass:A1} \\
    \max_{t=1,\ldots, n} \|G_{t,n}(\beps_0) - G_{t,n}(\tilde{\beps}_{0,h}) \|_{L_q} &\leq \Theta_n (h+1)^{-\beta}, \tag{A.2} \label{ass:A2} \\
    \int_0^1\|G_{\lfloor vn\rfloor,n}(\beps_0) - G_{v}(\beps_0)\|_{L_2}\, dv &\to 0.\tag{A.3}\label{ass:A3} 
\end{align*}

For non-random matrices $g_{t,n}, g_u \in \R^{m\times d}$, and random matrices $\hat{g}_{t,n}\in\R^{m\times d}$, define

\begin{align*}
    \Lambda_n& =\sqrt{\sum_{t=1}^n \|\hat{g}_{t,n} - g_{t,n}\|_{L_2}^2} ,  \\
    \Psi_n &=\sum_{t=2}^n \|g_{t,n} - g_{t-1,n}\| , \\
    \Phi_n &= \max_{t=1,\ldots,n} \|g_{t,n}\| + \sup_{u\in[0,1]} \|g_u\|, 
\end{align*}
and formulate the assumption 
\begin{align*}
    \int_0^1 \|g_{\lfloor vn\rfloor,n} - g_v\|\, dv &\to 0. \tag{A.4} \label{ass:A4}
\end{align*}

Define the rate 
\begin{align*}
    \xi(q,\beta) = 
	\begin{cases}
		\frac{q-2}{6q-4}, 
		& \beta\geq \frac{3}{2},\, \beta> \frac{2q}{q+2}, \\
		\frac{\beta-1}{4\beta-2}, 
		& \beta\geq \frac{3}{2},\, \beta \leq \frac{2q}{q+2}, \\
		\frac{(\beta-1)(q-2)}{4q\beta-3q-2}, 
		& \beta<\frac{3}{2},\, \beta > \frac{2q}{q+2}, \\
		\frac{(\beta-1)^2}{2\beta^2-1-\beta}, 
		& \beta<\frac{3}{2},\, \beta\leq \frac{2q}{q+2},
	\end{cases}
	\qquad q>2,\, \beta>1.
\end{align*}
and the local long run covariance matrices
\begin{align*}
    \Sigma_{t,n}
    &= \sum_{h=-\infty}^\infty g_{t,n} \Cov[G_{t,n}(\beps_0), G_{t,n}(\beps_h)] g_{t,n}^T, \\
    \Sigma_u &= \sum_{h=-\infty}^\infty g_{u} \Cov[G_{u}(\beps_0), G_{u}(\beps_h)] g_{u}^T.
\end{align*}

\begin{theorem}\label{thm:multiplier-FCLT}
    Suppose that $\hat{g}_{t,n}$ is $\beps_{t-L}$-measurable for $L=L_n$, and let \eqref{ass:A1}, \eqref{ass:A2}, \eqref{ass:A3}, \eqref{ass:A4} hold with $\Theta_n, \Phi_n=\mathcal{O}(1)$ such that
    \begin{align*}
        n^{-\frac{1}{2}} \Lambda_n^2 + \Lambda_n L_n^{1-\beta} + n^{\frac{1}{q}-\frac{1}{2}} L_n \to 0 \\
        \left( \Gamma_n + \Psi_n \right)^{\frac{\beta-1}{2\beta}} \sqrt{\log(n)} n^{-\xi(q,\beta)} \to 0.
    \end{align*}
    Then
    \begin{align*}
        \frac{1}{\sqrt{n}} \sum_{t=1}^{\floor{un}} \hat{g}_{t,n} [X_{t,n}-\E(X_{t,n})] \wconv \int_0^u \Sigma_v^\frac{1}{2}\, dW_v.
    \end{align*}
\end{theorem}

\section{Proofs}\label{sec:proofs}

\subsection{Preliminaries}

For $H\in(0,1)$, define 
\begin{align*}
    \gamma_H(h) &= |h+1|^{2H}-2|h|^{2H}+|h-1|^{2H}, h\in \R,\\
    \Gamma_H(h) &= -\gamma_H(h+1) +2\gamma_H(h) - \gamma_H(h-1).
\end{align*}
Note that $\gamma_H(h)$ is the autocovariance function of the increments of a fractional Brownian motion with Hurst parameter $H$, and $\Gamma_H(h)$ is the autocovariance of the corresponding second order increments.
In particular, if $(\sigma_s, H_s) \equiv (\sigma,H)$, then $\Cov(\chi_{i,n}, \chi_{j,n}) = n^{-2H}\sigma^2 \Gamma_H(i-j)$.
Moreover, we note that $\Gamma_H(h) \asymp |h|^{2H-4}$ as $|h|\to\infty$.

We proceed to give quantitative bounds on the approximation error between the nonstationary process and its stationary tangent process.
The following Lemma is in analogy to \cite[Lemma 1]{Coeurjolly2005} for the classical mBm.
The notable difference is that we do not require $\eta>\sup_v H_v$, which is an advantage of the Itô-mBm model.

\begin{lemma}\label{lem:bias}
    Suppose that $v\mapsto \theta_v = (\sigma_v, H_v)$ is Hölder continuous with exponent $\eta\in(0,1]$ on some open interval $I=(a,b]\subset (-\infty,1]$, and that $0<\underline{H} \leq H_v \leq \overline{H} < 1$ for all $v\in(-\infty,1]$, and $\sigma_v^2 \leq \overline{\sigma}^2<\infty$.
    For any $\theta = (\sigma^2, H)$ and $\frac{i}{n}\in I$, it holds that
    \begin{align*}
        n^{2H} \E \chi_{i,n}^2 &= \sigma^2 \Gamma_H(0) + \mathcal{O}\left(\log(n) n^{-\eta}\right) + \mathcal{O}\left(\|\theta-\theta_{\frac{i}{n}}\| \log(n) n^{ 2\|\theta-\theta_{\frac{i}{n}}\|}\right).
    \end{align*}
    For any interval $I$, the bound holds uniformly for $\theta \in [0,\overline{\sigma}^2] \times [\underline{H},\overline{H}]$.
\end{lemma}

\begin{proof}[Proof of Lemma \ref{lem:bias}]
    Denote $t_{i} = \frac{i}{n}$, and $\alpha_s = H_s-\frac{1}{2}$.
    Then, by Itô's isometry,
    \begin{align*}
        \E \chi_{i,n}^2 
        &= \int_{-\infty}^{t_i} \sigma_s^2\left[ (t_i - s)_+^{\alpha_s} - 2 (t_{i-1}-s)_+^{\alpha_s} + (t_{i-2}-s)_+^{\alpha_s} \right]^2\, ds  \\
        &= n^{-1}\int_{-\infty}^{n \cdot t_i} \sigma_{s/n}^2\left[ (t_i - \tfrac{s}{n})_+^{\alpha_{s/n}} - 2 (t_{i-1}-\tfrac{s}{n})_+^{\alpha_{s/n}} + (t_{i-2}-\tfrac{s}{n})_+^{\alpha_{s/n}} \right]^2\, ds \\
        &= \int_{-\infty}^{i} n^{-2H_{s/n}} \sigma^2_{s/n} \left[ (i-s)_+^{\alpha_{s/n}} - 2(i-1-s)_+^{\alpha_{s/n}} + (i-2-s)_+^{\alpha_{s/n}} \right]^2\, ds \\
        &= \int_0^\infty g(v, \theta_{\frac{i}{n} - \frac{v}{n}})\, dv, \qquad \theta_v = (H_v, \sigma_v^2), \\
        g(v,\theta) 
        &= n^{-2H} \sigma^2 \left[v^{H-\frac{1}{2}} - 2(v-1)_+^{H-\frac{1}{2}} + (v-2)_+^{H-\frac{1}{2}} \right]^2.
    \end{align*}
    Now observe that
    \begin{align*}
        |\partial_\sigma g(v,\theta)|
        & \leq n^{-2H} [v^{2H-1} \wedge v^{2H-5}],\\
        |\partial_H g(v,\theta)|
        & \leq n^{-2H} (16\sigma^2) [v^{2H-1} \wedge v^{2H-5}]\,[\log(n) + |\log(v)|].
    \end{align*}
    We may thus bound
    \begin{align*}
        &\int_0^\infty |g(v, \theta) - g(v, \theta_{\frac{i}{n}-\frac{v}{n}})|\, dv \\
        &\leq \int_0^\infty |g(v, \theta) - g(v, \theta_{\frac{i}{n}})|\, dv \;+\; \int_0^\infty | g(v, \theta_{\frac{i}{n}}) - g(v, \theta_{\frac{i}{n}-\frac{v}{n}})| \, dv \\
        &= A_1 + A_2,
    \end{align*}
    where, for some small $q\in(0,1)$,
    \begin{align*}
        A_1
        &\leq C n^{-2(H\wedge H_{\frac{i}{n}})}  \left[ \log(n) |H - H_{\frac{i}{n}}| \;+\; |\sigma^2 - \sigma^2_{\frac{i}{n}}|  \right], \\
        A_2 
        &\leq \int_0^{n^{1-q}} | g(v, \theta_{\frac{i}{n}}) - g(v, \theta_{\frac{i}{n}-\frac{v}{n}})| \, dv \\
        &\qquad + \int_{n^{1-q}}^\infty | g(v, \theta_{\frac{i}{n}}) - g(v, \theta_{\frac{i}{n}-\frac{v}{n}})| \, dv .
    \end{align*}
    Using the local Hölder continuity of $v\mapsto \theta_v$, we find that
    \begin{align*}
        &\int_0^{n^{1-q}} | g(v, \theta_{\frac{i}{n}}) - g(v, \theta_{\frac{i}{n}-\frac{v}{n}})| \, dv \\
        &\leq C \log(n) \max_{s\in[0, n^{-q}]} n^{-2H_{\frac{i}{n}-s}} \int_0^{n^{1-q}} \left( \tfrac{v}{n} \right)^\eta |v^{2\underline{H}-1} \wedge v^{2\overline{H}-5}| |\log(v)|\, dv \\
        &\leq C \log(n) n^{-2H} n^{-\eta},
    \end{align*}
    because $n^{n^{-\eta q}} \to 1$.
    Moreover,
    \begin{align*}
        \int_{n^{1-q}}^\infty | g(v, \theta_{\frac{i}{n}}) - g(v, \theta_{\frac{i}{n}-\frac{v}{n}})| \, dv
        &\leq  C \int_{n^{1-q}}^\infty \sup_{H\in[\underline{H},\overline{H}]} n^{-2H} v^{2H-5}\, dv \\
        &= \int_{n^{1-q}}^n \left( \tfrac{v}{n} \right)^{2\underline{H}} v^{-5}\, dv + \int_{n}^\infty \left( \tfrac{v}{n} \right)^{2\overline{H}} v^{-5}\, dv \\
        &\leq C (n^{(1-q)(2\underline{H}-4) - 2\underline{H}} + n^{-4}) \\
        &\leq C n^{-4+\epsilon},
    \end{align*}
    for any $\epsilon>0$, by making $q>0$ sufficiently small. 
    Since $\eta\in(0,1]$ and $H\in(0,1)$, the choice $\epsilon = \frac{1}{2}$ yields $n^{-4+\epsilon} \ll \log(n) n^{-2H-\eta}$.
    
    Hence, we have shown that
    \begin{align*}
        \int_0^\infty g(v,\theta_{\frac{i}{n} - \frac{v}{n}})\, dv 
        &= \int_0^\infty g(v,\theta)\, dv + \mathcal{O}\left(\log(n) n^{-2H-\eta}\right) \\
        &\qquad + \mathcal{O}\left(\|\theta-\theta_{\frac{i}{n}}\| \log(n) n^{-2H + 2\|\theta-\theta_{\frac{i}{n}}\|}\right).
    \end{align*}
    To complete the proof, we observe that $\int_0^\infty g(v,\theta) \, dv = n^{-2H} \sigma^2 \Gamma_H(0)$. 
\end{proof}

\begin{lemma}\label{lem:cov}
    Suppose that $v\mapsto \theta_v = (\sigma_v^2, H_v)$ is Hölder continuous with exponent $\eta\in(0,1]$ on some interval $I=(a,b]\subset (-\infty,1]$, and that $0<\underline{H} \leq H_v \leq \overline{H} < 1$ for all $v$, and $\sigma_v^2 \leq \overline{\sigma}^2<\infty$.
    Then for all $1\leq i \leq j \leq n$, such that $\frac{i}{n}\in I$,
    \begin{align*}
        \Cov(\chi_{i,n}, \, \chi_{j,n}) 
        & = n^{-2H_{i/n}} \sigma_{i/n}^2 \Gamma_{H_{i/n}}(i-j)  \\
        & \qquad + \mathcal{O}\left( \log(n) n^{-2H_{i/n}-(\eta \wedge \frac{1}{2})} (|i-j|\vee 1)^{H_{i/n}-\frac{5}{2}} \right). 
    \end{align*}
    For any interval $I$, the bound holds uniformly for $\theta \in [0,\overline{\sigma}^2] \times [\underline{H},\overline{H}]$.
\end{lemma}

\begin{proof}[Proof of Lemma \ref{lem:cov}]
    Denote $t_i=\frac{i}{n}$, and $\alpha_s = H_s-\frac{1}{2}$. 
    By Itô's isometry,
    \begin{align*}
        \Cov(\chi_{i,n}, \chi_{j,n}) 
        &= \int_{-\infty}^{\frac{i}{n}} \sigma_s^2 \left[ (t_i-s)^{\alpha_s} - 2(t_{i-1}-s)_+^{\alpha_s} + (t_{i-2}-s)^{\alpha_s} \right] \\
        &\qquad\qquad \cdot \left[ (t_j-s)^{\alpha_s} - 2(t_{j-1}-s)_+^{\alpha_s} + (t_{j-2}-s)^{\alpha_s} \right] \, ds \\
        &= \int_{-\infty}^{i} n^{-1} \sigma_s^2  \left[ (\tfrac{i}{n}-\tfrac{s}{n})^{\alpha_s} - 2(\tfrac{i-1}{n}-\tfrac{s}{n})_+^{\alpha_s} + (\tfrac{i-2}{n}-\tfrac{s}{n})^{\alpha_s} \right] \\
        &\qquad\qquad \cdot \left[ (\tfrac{j}{n}-\tfrac{s}{n})^{\alpha_s} - 2(\tfrac{j-1}{n}-\tfrac{s}{n})_+^{\alpha_s} + (\tfrac{j-2}{n}-\tfrac{s}{n})^{\alpha_s} \right] \, ds \\
        &= \int_{0}^{\infty} n^{-2H_{\frac{i-v}{n}}}\sigma_s^2 \left[ v^{\alpha_\frac{i-v}{n}} - 2(v-1)_+^{\alpha_\frac{i-v}{n}} + (v-2)^{\alpha_\frac{i-v}{n}} \right] \\
        &\qquad\qquad \cdot \left[ (v+\delta)^{\alpha_\frac{i-v}{n}} - 2(v+\delta-1)_+^{\alpha_\frac{i-v}{n}} + (v+\delta-2)^{\alpha_\frac{i-v}{n}} \right] \, dv \\
        &=: \int_0^\infty f(v, \delta, \theta_{\frac{i-v}{n}})\, dv,
    \end{align*}
    for $\delta = j-i$.
    We observe that, for $\delta \geq 1$,
    \begin{align*}
        \|D_\theta f(v,\delta,\theta)\| 
        &\leq C \log(n) n^{-2H} \left[ v^{H-\frac{1}{2}} \wedge v^{H-\frac{5}{2}} \right]\cdot (v+\delta)^{H-\frac{5}{2}} (1+|\log(v)|) \\
        &\leq  C \log(n) n^{-2H} (\delta\vee v)^{H-\frac{5}{2}} \left[ v^{H-\frac{1}{2}} \wedge v^{H-\frac{5}{2}} \right]
        =: C \log(n) \overline{f}(v,\delta, H).
    \end{align*}
    Using the $\eta$-Hölder continuity of $v\mapsto \theta_v$, with constant $c$, say, we find that
    \begin{align}
        & \left|\Cov(\chi_{i,n}, \chi_{j,n}) - \int_0^\infty f(v,\delta, \theta_{\frac{i}{n}})\,dv \right| \nonumber \\
        &\leq C\log(n) \int_0^\infty \left[ \left(\tfrac{v}{n}\right)^\eta \wedge 1 \right] \sup_{\substack{H\in[\underline{H},\overline{H}]\\ |H - H_{i/n}|\leq c (v/n)^\eta}} \overline{f}(v,\delta, H)\, dv. \label{eq:cov-1}
    \end{align}
    We split the domain of integration into the three segments $[0,1]$, $[1, n^{1-q}]$, and $[n^{1-q},\infty)$, for some small $q\in(0,1)$ to be specified later.
    
    \noindent
    The first portion of the integral may be bounded as
    \begin{align*}
        &\int_0^1 \left[ \left(\tfrac{v}{n}\right)^\eta \wedge 1 \right] \sup_{\substack{H\in[\underline{H},\overline{H}]\\ |H - H_{i/n}|\leq c (v/n)^\eta}} \overline{f}(v,\delta, H)\, dv \\
        &\leq n^{-\eta - 2H_{i/n} + cn^{-\eta}} \int_0^1 v^{\underline{H}+\eta - \frac{1}{2}} \delta^{H_{i/n}+ cn^{-\eta}-\frac{5}{2} } (1+|\log(v)|)\, dv \\
        &\leq C n^{-\eta - 2H_{i/n}} \;\delta^{H_{i/n}+cn^{-\eta}-\frac{5}{2}}.
    \end{align*}
    The second portion of the integral may be bounded as
    \begin{align*}
        &\int_1^{n^{1-q}} \left[ \left(\tfrac{v}{n}\right)^\eta \wedge 1 \right] \sup_{\substack{H\in[\underline{H},\overline{H}]\\ |H - H_{i/n}|\leq c (v/n)^\eta}} \overline{f}(v,\delta, H)\, dv \\
        &\leq n^{-\eta - 2H_{i/n} + cn^{-q\eta}} \int_1^{n^{1-q}} v^{H_{i/n}+\eta - \frac{5}{2}+cn^{-q\eta}} \delta^{H_{i/n}-\frac{5}{2} + cn^{-q\eta}} (1+|\log(v)|)\, dv \\
        &\leq n^{-\eta - 2H_{i/n}} \delta^{H_{i/n}+ c n^{-q\eta}-\frac{5}{2} } \left[ 1 + n^{(1-q)(H_{i/n}+\eta-\frac{3}{2}+cn^{-q\eta})} \right] \\
        &\leq n^{-\eta - 2H_{i/n}} \delta^{H_{i/n} + c n^{-q\eta}-\frac{5}{2} } \left[ 1 + n^{(1-q)(H_{i/n}+\eta-\frac{3}{2})} \right].
    \end{align*}
    The third portion of the integral may be bounded as
    \begin{align*}
        &\int_{n^{1-q}}^\infty \sup_{\substack{H\in[\underline{H},\overline{H}]\\ |H - H_{i/n}|\leq c (v/n)^\eta}} \overline{f}(v,\delta, H)\, dv \\
        &\leq \int_{n^{1-q}}^\infty  \sup_{H\in[\underline{H},\overline{H}]} \overline{f}(v,\delta, H)\, dv \\ 
        &= \int_{n^{1-q}}^\infty \sup_{H\in[\underline{H},\overline{H}]} n^{-2H} v^{H-\frac{5}{2}}  v^{H-\frac{5}{2}} \, dv \\
        &= \int_{n^{1-q}}^\infty \sup_{H\in[\underline{H},\overline{H}]} (\tfrac{v}{n})^{2H} v^{-5} \, dv \\
        &\leq \int_{n^{1-q}}^\infty (\tfrac{v}{n})^{2\underline{H}} v^{-5} \, dv
        \;+\;  \int_{n^{1-q}}^\infty (\tfrac{v}{n})^{2\overline{H}} v^{-5} \, dv \\
        &\leq C n^{(1-q)(2\underline{H}-4)-2\underline{H}} + C n^{(1-q)(2\overline{H}-4)-2\overline{H}} \\
        & \leq C n^{-4(1-q)} 
        \qquad \leq C n^{-4(1-q)} \left(\tfrac{\delta}{n}\right)^{H_{i/n}-\frac{5}{2}}
        \qquad \leq C \delta^{H_{i/n}-\frac{5}{2}} n^{-2H_{i/n}-\frac{1}{2}}.
    \end{align*}
    The last inequality holds for $q>0$ sufficiently small.
    We have thus established that
    \begin{align*}
        & \left|\Cov(\chi_{i,n}, \chi_{j,n}) - \int_0^\infty f(v,\delta, \theta_{\frac{i}{n}})\, dv \right| \\
        &\leq C \log(n) n^{-\eta - 2H_{i/n}} \delta^{H_{i/n}-\frac{5}{2} + cn^{-q\eta}} \left[1+n^{(1-q)(H_{i/n}+\eta-\frac{3}{2})}\right] + C \delta^{H_{i/n}-\frac{5}{2}} n^{-2H_{i/n}-\frac{1}{2}} \\
        &\leq  C \log(n)\delta^{H_{i/n}-\frac{5}{2}} n^{-2H_{i/n}} n^{-(\eta\wedge \frac{1}{2})}.
    \end{align*}
    
    To conclude the proof, we observe that $\int_0^\infty f(v,\delta,\theta)\, dv$ is the lag-$\delta$ autocovariance of the second order increments at frequency $\frac{1}{n}$ of a fractional Brownian motion with parameters $\theta = (\sigma^2,H)$.
    Thus, 
    \begin{align*}
        \int_0^\infty f(v,\delta,\theta)\, dv 
        &= n^{-2H}\sigma^2\,\Gamma_H(i-j).
    \end{align*}
    In the derivations above, we assumed that $i>j$. 
    For $i=j$, the claim is a direct consequence of Lemma \ref{lem:bias}, with $\theta=\theta_{i/n}$.
\end{proof}

\begin{lemma}\label{lem:longrunvar}
    Let the conditions of Lemma \ref{lem:cov} hold, and define $Z_{i,n}=(\chi_{i,n}^2, \, (\tilde{\chi}_{i,n}^2)^T$.
    Then
    \begin{align*}
        \Cov(Z_{i,j}, Z_{j,n}) 
        &= n^{-4H_{i/n}} \sigma_{i/n}^4 \Sigma_{H_{i_n}}(i-j) + \mathcal{O}\left( \log(n) n^{-4H_{i/n}-(\eta \wedge \frac{1}{2})} (|i-j|\vee 1)^{2H_{i/n}-5} \right) \\
        &= \mathcal{O}\left(  n^{-4H_{i/n}} (|i-j|\vee 1)^{-3} \right),
    \end{align*}
    where
    \begin{align*}
        \Sigma_H(h)
        &:=2\begin{pmatrix}
            \Gamma_H(h)^2 & (\Gamma_H(h) + \Gamma_H(h+1))^2 \\  ( \Gamma_H(h) + \Gamma_H(h-1) )^2 & (2\Gamma_H(h) + \Gamma_H(h-1) + \Gamma_H(h+1))^2
        \end{pmatrix}, \quad h\in\Z.
    \end{align*}
    The bound holds uniformly for $\theta \in [0,\overline{\sigma}^2] \times [\underline{H},\overline{H}]$.
\end{lemma}

\begin{proof}[Proof of Lemma \ref{lem:longrunvar}]
    Define $Y_{i,n} = (\chi_{i,n},\, \chi_{i,n}+2\chi_{i-1,n}+\chi_{i-2,n})^T = (\chi_{i,n},\, \tilde{\chi}_{i,n})^T$, and introduce the matrix.
    \begin{align*}
        \overline{\Gamma}_H(h) := \begin{pmatrix}
            \Gamma_H(h) & \Gamma_H(h) + \Gamma_H(h+1) \\  \Gamma_H(h) + \Gamma_H(h-1) & 2\Gamma_H(h) + \Gamma_H(h-1) + \Gamma_H(h+1)
        \end{pmatrix}.
    \end{align*}
    Via Lemma \ref{lem:cov}, we find that
    \begin{align*}
        \Cov(Y_{i,n}, Y_{j,n}) = n^{-2H_{i/n}} \sigma_{i/n}^2 \overline{\Gamma}_{H_{i/n}}(i-j) + \mathcal{O}\left( \log(n) n^{-2H_{i/n}-(\eta \wedge \frac{1}{2})} (|i-j|\vee 1)^{H_{i/n}-\frac{5}{2}} \right).
    \end{align*}
    Note that the matrix $\Sigma_H(h)$ is twice the entry-wise square of $\overline{\Gamma}_H(h)$, and that $|\overline{\Gamma}_H(h)| \asymp h^{2H-4}$.
    Hence, Lemma \ref{lem:cov-squared} yields 
    \begin{align*}
        \Cov(Z_{i,j}, Z_{j,n}) &= n^{-4H_{i/n}} \sigma_{i/n}^4 \Sigma_{H_{i_n}}(i-j) + \mathcal{O}\left( \log(n) n^{-4H_{i/n}-(\eta \wedge \frac{1}{2})} (|i-j|\vee 1)^{2H_{i/n}-5} \right).
    \end{align*}
\end{proof}

\begin{lemma}\label{lem:cov-squared}
    For two centered, jointly Gaussian random variables $X,Y$, it holds $\Cov(X^2, Y^2) = \Cov(X,Y)^2 \frac{\Var(X^2)}{\Var(X)^2} = 2\Cov(X,Y)^2$.
\end{lemma}

\begin{proof}
    Write $(X,Y) = (X, aX+Z)$ for $a=\Cov(X,Y)$, and $Z$ centered Gaussian, independent form $X$.
    Observe that for independent centered Gaussian random variables $X,Z$, and $a\in\R$, we have $\Cov(X^2, (aX+Z)^2) = a^2\Var(X^2) + 2a\E(X^3Z)+\Cov(X^2,Z^2) = a^2\Var(X^2)$, and $\Cov(X, aX+Z)=a\Var(X)$. 
\end{proof}

\begin{lemma}\label{lem:cov-log}
    There exists a universal $K$ such that for any two centered, jointly Gaussian random variables $X,Y$, with $\mathrm{Cor}(X,Y)=\rho$, \[\left|\Cov(\log(X^2), \log(Y^2))\right| \leq K |\rho| .\]
\end{lemma}

\begin{proof}
    Since this specific covariance is invariant to rescaling of $X$ and $Y$, we may assume both are standard normal, with correlation $\rho$. 
    We proceed similar to the proof of \cite[Lemma~8.5]{shen2020}.

    Denote by $H_l:\R\to\R$ the $l$-th Hermite polynomial (i.e.\ of degree $l$), and decompose $\log(x^2) = \sum_{l=0}^\infty c_l H_l(x)$. 
    Because $\log(X^2)\in L_2(P)$, the sequence $c_l$ is square-summable.
    Moreover,
    \begin{align*}
        \Cov(\log(X^2), \log(Y^2)) 
        &= \sum_{k,l=0}^\infty c_l c_k \Cov(H_l(X), H_k(Y)).
    \end{align*}
    Now use $\Cov(H_l(X), H_k(Y)) = \rho^k k! \mathds{1}(l=k)$ \citep[S.3.4]{shen2020} to find that
    \begin{align*}
        \Cov(\log(X^2), \log(Y^2)) 
        &= \sum_{k=1}^\infty c_k^2 \rho^k k!.
    \end{align*}
    Because $\Cov(\log(X^2), \log(X^2)) < \infty$, corresponding to $\rho=1$, we conclude that $\sum_k c_k^2 k! <\infty$.
    This yields $|\Cov(\log(X^2), \log(Y^2))| \leq |\rho| \sum_{k} c_k^2 k!$.
\end{proof}

\subsection{Local nonparametric estimation}

The error of $\widehat{\phi}_n(u)$ admits the following asymptotic representation.

\begin{proposition}\label{prop:NW}
    Suppose that $v\mapsto \theta_v = (\sigma_v, H_v)$ is Hölder continuous with exponent $\eta\in(0,l+1]$, and that $0<\underline{H} \leq H_v \leq \overline{H} < 1$ for all $v$, and $\sigma_v^2 \leq \overline{\sigma}^2<\infty$.
    If $b_n \ll n^{q-1}$ for some $q\in(0,1)$, then for any $p\geq 2$,
    \begin{align*}
        n^{2H_u} \widehat{\phi}_n(u) = \sigma_u^2 \Gamma_{H_u}(0) \cdot \begin{pmatrix}1 \\ 2^{2H_u} \end{pmatrix} + \mathcal{O}\left( \log(n)^{\lceil \eta\rceil} b_n^{\eta} \right) + \mathcal{O}_{L_p}\left( \frac{1}{\sqrt{n \,b_n}} \right) + \mathcal{O}\left( \log(n) n^{-(1\wedge \eta)} \right)
    \end{align*}
    For any $p$, the bound holds uniformly for $\theta \in [0,\overline{\sigma}^2] \times [\underline{H},\overline{H}]$, and for all $(b_n,u)$ such that $u\in[b_n,1-b_n]$.
\end{proposition}

\begin{proof}[Proof of Proposition \ref{prop:NW}]
    Lemma \ref{lem:bias} yields
    \begin{align*}
        \E\left(  \widehat{\phi}_{n,1}(u) \right) 
        = \sum_{i=1}^n w_{i,n} \E \chi_{i,n}^2 
        &= \sum_{i=1}^n w_{i,n} n^{-2H_{i/n}} \sigma_{i/n}^2 \Gamma_{H_{i/n}}(0)  +  \mathcal{O}\left(\log(n) n^{-(1\wedge\eta)-2H_u}\right), \\
        \leadsto \quad \E\left( n^{2H_u} \widehat{\phi}_{n,1}(u) \right)
        &= \sum_{i=1}^n w_{i,n} n^{2(H_u-H_{i/n})} \sigma_{i/n}^2 \Gamma_{H_{i/n}}(0)  +  \mathcal{O}\left(\log(n) n^{-(1\wedge\eta)}\right) .
    \end{align*}
    A Taylor expansion of the function $\mu_n(s)= n^{2(H_u-H_s)} \sigma_s^2 \Gamma_{H_s}(0)$, together with property (iv) of the weights, yields,
    \begin{align*}
        \E\left( n^{2H_u} \widehat{\phi}_{n,1}(u) \right) 
        &= \mu_n(u) + \mathcal{O}\left(b_n^\eta \log(n)^{\lceil \eta\rceil}\right) + \mathcal{O}\left( \log(n) n^{-(1\wedge \eta)} \right)
    \end{align*}
    The same bound holds for $\widehat{\phi}_{n,2}$, hence
    \begin{align*}
        \E\left( n^{2H_u} \widehat{\phi}_n(u) \right) &= \sigma_u^2 \Gamma_{H_u}(0) \cdot \begin{pmatrix}1 \\ 2^{2H_u} \end{pmatrix} +  \mathcal{O}\left(\log(n) n^{-(1\wedge\eta)}\right) + \mathcal{O}\left(b_n^\eta \log(n)^{\lceil \eta\rceil}\right),
    \end{align*}
    and the bound holds uniformly in $u\in(0,1)$.
    
    To bound the stochastic term, denote $Z_{i,n} = (\chi_{i,n}^2, \tilde{\chi}_{i,n}^2)^T$ and observe
    \begin{align*}
         \left\| \Cov\left(n^{2H_u} \widehat{\phi}_{n}(u)\right) \right\| 
        &\leq  n^{4H_u}\sum_{i,j=1}^n w_{i,n}(u) w_{j,n}(u)\left\| \Cov(Z_{i,n}, Z_{j,n}) \right\| \\
        &\leq \frac{C^2 n^{4H_u}}{n^2 \, b_n^2} \sum_{i=1}^n \sum_{j=1}^n \mathds{1}_{|\frac{i}{n}-u|\leq b_n} \mathds{1}_{|\frac{j}{n}-u|\leq b_n} \|\Cov(Z_{i,n}, Z_{j,n})\|.
    \end{align*}
    For $|\frac{i}{n}-u|\leq b_n \ll n^{1-q}$, and $|\frac{j}{n}-u|\ll n^{1-q}$, by virtue of Lemma \ref{lem:longrunvar},
    \begin{align*}
        \|\Cov(Z_{i,n}, Z_{j,n}) \|
        &\leq C (|i-j|\vee 1)^{-3} n^{-4H_{u} + C n^{-q}}\\
        &\leq C (|i-j|\vee 1)^{-3} n^{-4H_{u}}.
    \end{align*}
    Thus,
    \begin{align*}
        \left\|\Cov\left(n^{2H_u} \widehat{\phi}_{n}(u)\right) \right\|
        &\leq C \frac{1}{n^2b_n^2} \sum_{i,j=1}^n \mathds{1}_{|\frac{i}{n}-u|\leq b_n} \mathds{1}_{|\frac{j}{n}-u|\leq b_n} (|i-j|\vee 1)^{-3} \\
        &\leq C \frac{1}{n^2b_n^2} \sum_{i=\lfloor n(u-b_n)\rfloor \vee 1}^{\lceil n(u+b_n)\rceil \wedge n} \sum_{j=-\infty}^\infty (|i-j|\vee 1)^{-3} 
        \qquad \leq \frac{C}{n\, b_n}.
    \end{align*}
    This establishes the upper bound on the stochastic term of $\widehat{\phi}_{n}(u)$ in $L_2$. 
    To obtain the bound in $L_p$, we observe that $\widehat{\phi}_n(u)$ belongs to the second Wiener chaos, such that all $L_p$ norms are equivalent \citep[2.8.14]{nourdin2012a}. 
\end{proof}

\begin{proof}[Proof of Theorem \ref{thm:rate}]
    Proposition \ref{prop:NW} yields
    \begin{align*}
        \widehat{H}_n(u) 
        &= \frac{1}{2} \log_2 \left(  \frac{2^{2H_u}\sigma^2 \Gamma_{H_u}(0) + \mathcal{O}(\log(n)^{\lceil \eta\rceil}b_n^\eta) + \mathcal{O}_{L_p}(1/\sqrt{n b_n}) + \mathcal{O}(\log(n) n^{-(1\wedge \eta)})  }{\sigma^2 \Gamma_{H_u}(0) + \mathcal{O}(\log(n)^{\lceil \eta\rceil}b_n^\eta) + \mathcal{O}_{L_p}(1/\sqrt{n b_n}) + \mathcal{O}(\log(n) n^{-(1\wedge \eta)})} \right) \\
        &= \frac{1}{2} \log_2 \left( 2^{2H_u}  + \mathcal{O}(\log(n)^{\lceil \eta\rceil}b_n^\eta) + \mathcal{O}_{L_p}(1/\sqrt{n b_n}) +\mathcal{O}(\log(n) n^{-(1\wedge \eta)}) \right) \\
        &= H_u +  \mathcal{O}(\log(n)^{\lceil \eta\rceil}b_n^\eta) + \mathcal{O}_{L_p}(1/\sqrt{n b_n}) +\mathcal{O}(\log(n) n^{-(1\wedge \eta)})
        \intertext{and for the optimal choice of $b_n$,}
        &= H_u + \mathcal{O}_{L_p}\left( \log(n)^{\frac{\lceil \eta\rceil}{2\eta+1}} n^{-\frac{\eta}{2\eta+1}} \right).
    \end{align*}
\end{proof}
\begin{proof}[Proof of Theorem \ref{thm:rate2}]
    Denote $a_{i,n}=\E \chi_{i,n}^2$ and $\tilde{a}_{i,n} = \E \widetilde{\chi}_{i,n}^2$. 
    Then $\chi_{i,n}^2 \deq a_{i,n} Z^2$ for $Z\sim \mathcal{N}(0,1)$, and $\widetilde{\chi}_{i,n}^2 \deq \tilde{a}_{i,n} Z^2$, so that Lemma \ref{lem:bias} yields
    \begin{align*}
        \E \log_2\left( \tfrac{\widetilde{\chi}_{i,n}^2}{\chi_{i,n}^2} \right) 
        &= \log_2 \left(\tfrac{\tilde{a}_{i,n}}{a_{i,n}}\right) + \left[ \E \log_2(Z^2) - \E \log_2(Z^2)\right] 
        = \log_2 \left(\tfrac{\tilde{a}_{i,n}}{a_{i,n}}\right)\\
        &= 2 H_{i/n}+ \log_2 \left(\tfrac{(n/2)^{2H_{i/n}}\tilde{a}_{i,n}}{n^{2H_{i/n}}a_{i,n}}\right) \\
        &= 2H_{i/n} + \log_2 \left( \frac{\sigma^2 \Gamma_{H_{i/n}}(0) + \mathcal{O}(\log(n) n^{-(\eta \wedge 1)})}{\sigma^2 \Gamma_{H_{i/n}}(0) + \mathcal{O}(\log(n) n^{-(\eta \wedge 1)})} \right) \\
        &= 2H_{i/n} + \log_2 \left( 1 + \mathcal{O}(\log(n) n^{-(\eta \wedge 1))} \right) \\
        &= 2 H_{i/n} + \mathcal{O}\left( \log(n) n^{-(\eta \wedge 1)} \right).
    \end{align*}
    Thus, by standard bias bounds for local polynomial estimators, we obtain the bias bound $\E(\widehat{H}_n^{\dagger}(u)) = H(u) + \mathcal{O}(b_n^{\eta} + \log(n) n^{-(\eta\wedge 1)})$.

    To control the variance, we use Lemma \ref{lem:cov-log} to find that
    \begin{align*}
        \Var \left( \sum_{i=1}^n w_{i,n}(u) \log_2(\chi^2_{i,n}) \right) 
        & \leq K \sum_{i,j=1}^n |w_{i,n}(u) w_{j,n}(u)| \cdot \left|\text{Cor}(\chi_{i,n}, \chi_{j,n}) \right| .
    \end{align*}
    In the following, $K$ is a generic constant which may vary from line to line.
    Lemma \ref{lem:cov} yields
    \begin{align*}
        \left|\text{Cor}(\chi_{i,n}, \chi_{j,n}) \right|
        &\leq K|i-j|^{2\overline{H}-4} + K |i-j|^{\overline{H}-\frac{5}{2}} \quad \leq K |i-j|^{-\frac{3}{2}},
    \end{align*}
    which is summable.
    In combination with the boundedness and finite support of the weights $w_{i,n}(u)$, we find 
    \begin{align*}
        \Var \left( \sum_{i=1}^n w_{i,n}(u) \log_2(\chi^2_{i,n}) \right) 
        &\leq K \sum_{i,j=\floor n(u-b_n)\rfloor}^{\lceil n(u+b_n)\rceil} w_{i,n}(u) w_{j,n}(u) (|i-j|+ 1)^{-\frac{3}{2}} \\
        &\leq \frac{K}{n b_n} \sum_{h=-\infty}^\infty (|h|+1)^{-\frac{3}{2}} \quad = \mathcal{O}(1/(nb_n)).
    \end{align*}
    Similarly, $\Var \left( \sum_{i=1}^n w_{i,n}(u) \log_2(\chi^2_{i,n}) \right) = \mathcal{O}(1/(nb_n))$, which in particular yields
    \begin{align*}
        \Var \left(\widehat{H}^\dagger_n(u)\right) = \mathcal{O}\left(\tfrac{1}{nb_n}\right).
    \end{align*}
\end{proof}

\subsection{Integrated parameter estimation}

\begin{proof}[Proof of Theorem \ref{thm:integrated-clt}]
    \underline{Step (i):}
    Denote $t_0 = 2L_n$.
    We may write the estimator $\widehat{\mathcal{H}}(u)$ equivalently as
    \begin{align*}
        \widehat{\mathcal{H}}(u)
        &= \frac{1}{n} \sum_{t=2L}^{\lfloor un\rfloor } \left\{ m(\hat{\psi}_{t,n} ) + Dm(\hat{\psi}_{t,n} ) \cdot (  Z_{t,n} - \hat{\psi}_{t,n}  )   \right\}, 
    \end{align*}
    with 
    \begin{align*}
        m&:(0,\infty)^2\to \R, \;(x,y) \mapsto \left[\tfrac{1}{2} \log(\tfrac{y}{x})\right]\;\vee 0 \;\wedge 1, \\
        Z_{t,n} &= \frac{n^{2H_{t/n}}}{\sigma_{t/n}^2}\begin{pmatrix}
            \chi_{t,n}^2 \\ \tilde{\chi}_{t,n}^2
        \end{pmatrix}, \qquad \psi_{u}= \Gamma_{H_u}(0) \cdot \begin{pmatrix}
            1 \\ 2^{2H_{u}}
        \end{pmatrix}, \qquad
        \psi_{t,n} = \psi_{t/n} \\
        \hat{\psi}_{t,n} &= \frac{n^{2H_{t/n}}}{\sigma_{t/n}^2} \widehat{\phi}(\tfrac{t-L_n}{n}) \text{ for } t=2L,\ldots, n, \text{ and } \hat{\psi}_{t,n}=\psi_{t,n} \text{ for } t=1,\ldots, 2L-1.
    \end{align*}
    Due to the $L_q$ bound of Lemma \ref{prop:NW}, we have for any $q\geq 2$ and $t= L+\lceil n b_n\rceil,\ldots, n$,
    \begin{align*}
        \|\hat{\psi}_{t,n}-\psi_{t-L,n}\|_{L_q}^2 
        &= \mathcal{O}\left(\log(n)^{2\lceil \eta\rceil} b_n^{2\eta} + \tfrac{1}{n b_n} + \log(n)^2 n^{-2(1\wedge \eta)}\right)
        \;\leq\;\mathcal{O}\left(\log(n)^{2\lceil \eta\rceil} n^{-\frac{1}{2}-r}\right).
    \end{align*} 
    Furthermore, for $t=2L,\ldots, L+\lceil n b_n\rceil$, the local polynomial estimator effectively uses the smaller bandwidth $\frac{t-L}{n}\geq \frac{L}{n}$, see Remark \ref{rem:boundary}.
    Hence
       \begin{align*}
        \|\hat{\psi}_{t,n}-\psi_{t-L,n}\|_{L_q}^2 
        &= \mathcal{O}\left(\log(n)^{2\lceil \eta\rceil} (\tfrac{t-L_n}{n})^{2\eta} + \tfrac{1}{t-L_n} + \log(n)^2 n^{-2(1\wedge \eta)}\right).
    \end{align*} 
    Moreover, $\|\psi_{t-L,n} - \psi_{t,n}\|^2 = \mathcal{O}((L_n/n)^{2\eta})$, and hence $\| \max_{t=2L,\ldots, n} |\hat{\psi}_{t,n}-\psi_{t,n} | \|_{L_q} = \mathcal{O}(n^{\frac{1}{q} - \epsilon}) + \mathcal{O}((L_n/n)^{\eta})$, for some $\epsilon>0$, which tends to zero for $q$ large enough.
    Using the assumptions on $b_n$, we find
    \begin{align*}
        \frac{1}{n} \sum_{t=2L}^n \|\hat{\psi}_{t,n} - \psi_{t,n}\|_{L_q}^2 
        & \leq C \left( \log(n)^{2\lceil \eta\rceil} n^{-\frac{1}{2}-r} \right) 
        + \frac{C}{n} \sum_{t=2L}^{L+\lceil n b_n \rceil} \left[\log(n)^{2\lceil \eta\rceil} \left(\tfrac{t-L}{n}\right)^{2\eta} + \frac{1}{t-L} \right] \\
        &\leq C \left( \log(n)^{2\lceil \eta\rceil} n^{-\frac{1}{2}-r} \right) 
        + C b_n^{1+2\eta} \log(n)^{2\lceil \eta\rceil} + C \frac{\log(n)}{n} \\
        &= \mathcal{O}\left( \log(n)^{2\lceil \eta\rceil} n^{-\frac{1}{2}-r}\right).
    \end{align*}
    Thanks to the uniform convergence, we may restrict our attention to the event
    \begin{align*}
        A_n=\left\{ \frac{1}{2} \psi_{t,n} \leq \hat{\psi}_{t,n} \leq 2 \psi_{t,n}, \qquad \forall \, t=1,\ldots,n \right\},
    \end{align*}
    as $P(A_n)\to 1$ as $n\to\infty$.
    For this event, a Taylor expansion of $m$ around $\hat{\psi}_{t,n}$ yields
    \begin{align}
        &\widehat{\mathcal{H}}(u) - \frac{1}{n} \sum_{t=1}^{\lfloor un\rfloor} m(\psi_{t,n}) \nonumber \\
        &=\frac{1}{n} \sum_{t=1}^{\lfloor un\rfloor } \left\{ m(\hat{\psi}_{t,n} ) + Dm(\hat{\psi}_{t,n} ) \cdot (  Z_{t,n} - \hat{\psi}_{t,n}  )   \right\} - \frac{1}{n}\sum_{t=1}^{\lfloor un\rfloor} m(\psi_{t,n}) + \mathcal{O}_P\left( \tfrac{L_n}{n} \right) \nonumber\\
        &= \frac{1}{n} \sum_{t=1}^{\lfloor un\rfloor} Dm(\hat{\psi}_{t,n})\cdot (Z_{t,n}-\psi_{t,n}) + \mathcal{O}_P\left(\frac{1}{n} \sum_{t=1}^n \|\hat{\psi}_{t,n} - \psi_{t,n}\|^2\right) + \mathcal{O}_P\left( \tfrac{L_n}{n} \right) \label{eqn:integrated-clt-12} \\
        &= \frac{1}{n} \sum_{t=1}^{\lfloor un\rfloor} Dm(\hat{\psi}_{t,n})\cdot (Z_{t,n}-\psi_{t,n}) + \mathcal{O}_P\left(\log(n)^{2\lceil \eta\rceil} n^{-\frac{1}{2}-r}\right) + \mathcal{O}_P\left( \tfrac{L_n}{n} \right) , \label{eqn:integrated-clt-12a} 
    \end{align}
    using in \eqref{eqn:integrated-clt-12} that the second derivative $D^2m(\psi)$ is bounded in the specified neighborhood of $\psi_{t,n}$, and the special definition $\hat{\psi}_{t,n} = \psi_{t,n}$ for $t=1,\ldots, 2L-1$. 
    By our assumption on $L_n$, the remainder terms in \eqref{eqn:integrated-clt-12a} are of order $o_P(1/\sqrt{n})$.
    Moreover, the approximation \eqref{eqn:integrated-clt-12} holds uniformly in $u\in[0,1]$.

    Next, note that the Hölder regularity of $H_v$ with exponent $\eta$ yields
    \begin{align*}
        \frac{1}{n} \sum_{t=1}^{\lfloor un\rfloor} m(\psi_{t,n}) 
        = \frac{1}{n} \sum_{t=1}^{\lfloor un\rfloor} H_{t/n} 
        = \int_0^u H_v\, dv + \mathcal{O}(n^{-\eta}),
    \end{align*}
    uniformly in $u\in[0,1]$. 
    Since $\eta>\frac{1}{2}$ by assumption, we have shown that
    \begin{align*}
        \sqrt{n}\left[\widehat{\mathcal{H}}(u) - \mathcal{H}(u)\right] = \frac{1}{\sqrt{n}} \sum_{t=1}^{\lfloor un\rfloor} Dm(\hat{\psi}_{t,n})\cdot (Z_{t,n}-\psi_{t,n}) + o_P(1).
    \end{align*}
    By Lemma \ref{lem:bias}, $|\E(Z_{t,n}) - \psi_{t,n}| = \mathcal{O}(\log(n) n^{-\eta}) = o(n^{-\frac{1}{2}})$, and we obtain
    \begin{align}
        \sqrt{n}\left[\widehat{\mathcal{H}}(u) - \mathcal{H}(u)\right] = \frac{1}{\sqrt{n}} \sum_{t=1}^{\lfloor un\rfloor} Dm(\hat{\psi}_{t,n})\cdot (Z_{t,n}-\E Z_{t,n}) + o_P(1). \label{eqn:integrated-clt-2}
    \end{align}

    \noindent\underline{Step (ii):}
    We now apply the multiplier FCLT of Theorem \ref{thm:multiplier-FCLT} to the leading term in \eqref{eqn:integrated-clt-2}, with 
    \begin{align*}
        \hat{g}_{t,n} = Dm(\hat{\psi}_{t,n}) \mathds{1}_{A_n}, \qquad g_{t,n} = Dm(\psi_{t,n}), \qquad g_u = Dm(\psi_u), \qquad
        Y_{t,n} = Z_{t,n} - \E Z_{t,n}.
    \end{align*}
    For this definition of $g_{t,n}$ and $\hat{g}_{t,n}$, \eqref{ass:A4} holds, $\Phi_n$ is bounded, and $\Psi_n=\mathcal{O}(n^{1-\eta})$.
    Moreover, Proposition \ref{prop:NW} shows that 
    \begin{align*}
        \Lambda_n^2
        &\leq C \sum_{t=2L}^n \left(\|\hat{\psi}_{t,n}-\psi_{t-L,n}\|_{L_2}^2 + \|\psi_{t-L,n}-\psi_{t,n}\|^2 \right) \\
        &\leq C\left( \log(n)^2n^{\frac{1}{2}-r} + L_n^{2\eta} n^{1-2\eta} \right) + C \sum_{t=2L_n}^{L_n+\lceil n b_n\rceil} \left( \log(n)^2 (\tfrac{t-L_n}{n})^{2\eta} + \tfrac{1}{t-L_n}\right) \\
        &\leq C\left( \log(n)^2 n^{\frac{1}{2}-r} + L_n n^{1-2\eta} + \log(n)^2 \tfrac{(nb_n)^{2\eta+1}}{n^{2\eta}} + \log(nb_n) \right) \\
        &= \mathcal{O}\left(  \log(n)^2 n^{\frac{1}{2}-r} + L_n n^{1-2\eta} \right)
    \end{align*}
    It remains to check the conditions \eqref{ass:A1}, \eqref{ass:A2}, \eqref{ass:A3} for $Y_{t,n}$.
    To write $Y_{t,n}$ in the form $G_{t,n}(\beps_t)$ as in Appendix \ref{sec:multiplier}, note that $Y_{t,n}$ is a functional of the driving Brownian motion $B_s$, see Definition \eqref{eqn:def-Ito}. 
    Since all Polish spaces are Borel-isomorphic, there exists a Borel-isomorphism $\varphi:(0,1)\to C[0,1]$, and iid random variables $\epsilon_t\sim U(0,1)$ such that $\varphi(\epsilon_t) = [\tilde{B}_u]_{u\in[0,1]} = \sqrt{n}[B_{\frac{u+t-1}{n}} - B_{\frac{t-1}{n}}]_{u\in[0,1]} = \varphi(\epsilon_t)$, for a standard Brownian motion $\tilde{B}$. 
    With this notation, we may write
    \begin{align*}
        \frac{n^{H_{i/n}}}{\sigma_{i/n}} \chi_{i,n} &=  \frac{n^{H_{i/n}}}{\sigma_{i/n}}\int_{-\infty}^{i/n} g_{i,n}(s) \, dB_s = \widetilde{G}_{i,n}(\beps_i), \\
        g_{i,n}(s) &= \sigma_s \left[ (\tfrac{i}{n}-s)_+^{H_s-\frac{1}{2}} - 2 (\tfrac{i}{n}-\tfrac{1}{n}-s)_+^{H_s-\frac{1}{2}} + (\tfrac{i}{n}-\tfrac{2}{n}-s)_+^{H_s-\frac{1}{2}} \right], \\
        \widetilde{G}_{i,n}(\beps_j) &= \frac{n^{H_{i/n}}}{\sigma_{i/n}} \int_{-\infty}^{i/n} g_{i,n}(s) \, dB_{s+\frac{j-i}{n}} \\
        &=  \frac{n^{H_{i/n}}}{\sigma_{i/n}}\sum_{k=0}^\infty \int_{(i-k-1)/n}^{(i-k)/n}  g_{i,n}(s)\, dB_{s+\frac{j-i}{n}} \quad
        =\frac{n^{H_{i/n}}}{\sigma_{i/n}} \sum_{k=0}^\infty \int_{i-k-1}^{i-k}  g_{i,n}(\tfrac{s}{n})\, \tfrac{1}{\sqrt{n}}d\tilde{B}_{s+(j-i)} \\
        &= \sum_{k=0}^\infty \int_{i-k-1}^{i-k} \tilde{g}_{i,n}(s)\, d\tilde{B}_{s+(j-i)} \\
        &= \sum_{k=0}^\infty \zeta_{i,k,n}(\epsilon_{j-k}), \qquad \zeta_{i,k,n}:(0,1)\to\R, \; \zeta_{i,k,n}(\epsilon_{j-k})\in L_2(P), \\
        \tilde{g}_{i,n}(s) &= \frac{\sigma_{\frac{s}{n}}}{\sigma_{\frac{i}{n}}} n^{H_{\frac{i}{n}}-H_{\frac{s}{n}}}\,\bar{g}(s-i, H_{\frac{s}{n}}), \\
        \bar{g}_{i}(s, H) &= (-s)_+^{H-\frac{1}{2}} - 2 (-1-s)_+^{H-\frac{1}{2}} + (-2-s)_+^{H-\frac{1}{2}}.
    \end{align*}
    In particular, the Bernoulli shift of the $\epsilon_t$ is equivalent to shifting the driving Brownian motion. 

    \noindent\underline{Ad \eqref{ass:A2}:}
    The $\zeta_{i,k,n}(\epsilon_j)$ are centered Gaussian random variables with variance
    \begin{align*}
        \|\zeta_{i,k,n}(\epsilon_j)\|_{L_2}^2 
        &= \int_{i-k-1}^{i-k} |\tilde{g}_{i,n}(s)|^2\, ds \\
        &\leq C n^{2(H_{\frac{i}{n}} - H_{\frac{i-k}{n}}) + C n^{-\eta}} (k+1)^{2H_{\frac{i-k}{n}} -5 + C n^{-\eta}}\\
        &\leq C n^{2(H_{\frac{i}{n}} - H_{\frac{i-k}{n}})} (k+1)^{2H_{\frac{i-k}{n}} -5 +\delta},
    \end{align*}
    for $\delta>0$ small enough, and $n$ large enough.
    If $H_{\frac{i}{n}} \leq H_{\frac{i-k}{n}}$, then $\|\zeta_{i,k,n}(\epsilon_j)\|_{L_2}^2 \leq C(k+1)^{2\overline{H}-5+\delta} \leq C (k+1)^{-3}$ for $\delta$ small enough, because $\overline{H}<1$. 
    If instead $H_{\frac{i}{n}} > H_{\frac{i-k}{n}}$, we further distinguish two cases to obtain
    \begin{align}
        \|\zeta_{i,k,n}(\epsilon_j)\|_{L_2}^2 
        &\leq 
        \begin{cases}
             C (k+1)^{\frac{2}{1-\delta}(H_{\frac{i}{n}} - H_{\frac{i-k}{n}})} (k+1)^{2H_{\frac{i-k}{n}} -5 +\delta}, & k\geq n^{1-\delta},  \\
             C n^{C n^{-\delta \eta}} (k+1)^{2H_{\frac{i-k}{n}} -5 +\delta},& k< n^{1-\delta},
        \end{cases}  \nonumber \\
        & \leq 
        \begin{cases}
            C (k+1)^{\frac{2}{1-\delta} H_{\frac{i}{n}} - \frac{2\delta}{1-\delta} H_{\frac{i-k}{n}} -5 +\delta}, & k\geq n^{1-\delta},  \\
             C  (k+1)^{2\overline{H} -5 +\delta},& k< n^{1-\delta},
        \end{cases} \nonumber \\
        & \leq 
        \begin{cases}
            C (k+1)^{\frac{2}{1-\delta} \overline{H} -5 +\delta}, & k\geq n^{1-\delta},  \\
             C  (k+1)^{2\overline{H} -5 +\delta},& k< n^{1-\delta},
        \end{cases} \nonumber
        \\
        &\leq C (k+1)^{-3}, \label{eqn:phys-dep-mbm}
    \end{align}
    for $\delta$ small enough. 
    Since $\zeta_{i,k,n}(\epsilon_j)$ is Gaussian, the latter $L_2$ bound also holds in $L_q$.
    Thus, we have shown that for all $q\geq 2$, and for some $C=C_q$,
    \begin{align*}
        \|\widetilde{G}_{i,n}(\beps_0) - \widetilde{G}_{i,n}(\tilde{\beps}_{0,h})\|_{L_q} \leq C\|\zeta_{i,h,n}(\epsilon_j)\|_{L_2} &\leq C (h+1)^{-\frac{3}{2}}, \\
        \|\widetilde{G}_{i,n}(\beps_0)\|_{L_q} &\leq C.
    \end{align*}
    Now note that $Y_{t,n}$ is a function of the $\widetilde{G}_{i,n}(\beps_i)$, in particular
    \begin{align}
        Y_{i,n} = G_{t,n}(\beps_j) = 
        \begin{pmatrix}
            \widetilde{G}_{t,n}(\beps_j)^2 \\
            (\widetilde{G}_{t,n}(\beps_j)+2\widetilde{G}_{t-1,n}(\beps_{j-1})+\widetilde{G}_{t-2,n}(\beps_{j-2}))^2
        \end{pmatrix},\label{eq:defG}
    \end{align}
    and it is straightforward to derive
    \begin{align}
        \|G_{t,n}(\beps_0) - G_{t,n}(\tilde{\beps}_{0,h})\|_{L_q}  \leq C (h+1)^{-\frac{3}{2}}. \label{eqn:kernel-decay}
    \end{align}
    This establishes \eqref{ass:A2} with $\Theta_n=\mathcal{O}(1)$ and exponent $\beta=\frac{3}{2}$.

    \noindent\underline{Ad \eqref{ass:A1}:}
    Observe that
    \begin{align}
        &\left\|\widetilde{G}_{i,n}(\beps_0) - \widetilde{G}_{i-1,n}(\beps_0)\right\|_{L_2}^2 \nonumber \\
        &= \left\| \sum_{k=0}^\infty \int_{i-k-1}^{i-k} \tilde{g}_{i,n}(s)\, d\tilde{B}_{s-i} - \sum_{k=0}^\infty \int_{(i-1)-k-1}^{(i-1)-k} \tilde{g}_{i-1,n}(s)\, d\tilde{B}_{s-(i-1)} \right\|_{L_2}^2 \nonumber \\
        &= \left\| \sum_{k=0}^\infty \int_{i-k-1}^{i-k} \left[ \tilde{g}_{i,n}(s) - \tilde{g}_{i-1,n}(s-1)\right]\, d\tilde{B}_{s-i} \right\|_{L_2}^2 \nonumber \\
        &= \sum_{k=0}^\infty \int_{i-k-1}^{i-k} |\tilde{g}_{i,n}(s) - \tilde{g}_{i-1,n}(s-1)|^2\, ds \nonumber \\
        & \leq \sum_{k=0}^\infty \int_{i-k-1}^{i-k} \left| \frac{\sigma_{\frac{s}{n}}}{\sigma_{\frac{i}{n}}} n^{H_{\frac{i}{n}}-H_{\frac{s}{n}}} - \frac{\sigma_{\frac{s}{n}}}{\sigma_{\frac{i-1}{n}}} n^{H_{\frac{i-1}{n}}-H_{\frac{s}{n}}} \right|^2 |\bar{g}(s-i,H_{\frac{s}{n}})|^2\, ds \nonumber \\
        &\quad +  \sum_{k=0}^\infty \int_{i-k-1}^{i-k} \left|  \frac{\sigma_{\frac{s}{n}}}{\sigma_{\frac{i-1}{n}}} n^{H_{\frac{i-1}{n}}-H_{\frac{s}{n}}} \right|^2 |\bar{g}(s-i, H_{\frac{s}{n}}) - \bar{g}(s-i, H_{\frac{s-1}{n}})|^2\, ds \nonumber \\
        &\leq  \sum_{k=0}^\infty \left|1 - \frac{\sigma_{\frac{i}{n}}}{\sigma_{\frac{i-1}{n}}} n^{H_{\frac{i-1}{n}} - H_{\frac{i}{n}}} \right| \int_{i-k-1}^{i-k}  |\tilde{g}_{i,n}(s)|^2\, ds \nonumber \\
        &\quad +  \sum_{k=0}^\infty \int_{i-k-1}^{i-k} \left|  n^{H_{\frac{i-1}{n}}-H_{\frac{s}{n}}} \right|^2 |\bar{g}(s-i, H_{\frac{s}{n}}) - \bar{g}(s-i, H_{\frac{s-1}{n}})|^2\, ds \nonumber \\
        \begin{split}
        &\leq C\log(n)\, n^{-2\eta} \sum_{k=0}^\infty (k+1)^{-3}   \\
        &\quad + C  \sum_{k=0}^\infty  \log(n) n^{-\eta} n^{2H_{\frac{i-1}{n}} - 2 H_{\frac{i-k}{n}}} \int_{i-k-1}^{i-k}  |\bar{g}(s-i, H_{\frac{s}{n}}) - \bar{g}(s-i, H_{\frac{s-1}{n}})|^2\, ds .
        \end{split}\label{eqn:intpar-TV-1}
    \end{align}
    To bound the latter integral, we exploit the fact that
    \begin{align*}
        \frac{d}{dH} \bar{g}(s,H) 
        &= \mathds{1}_{s<0} \log(-s) (-s)^{H-\frac{1}{2}} \\
        &\quad - 2 \mathds{1}_{s<-1} \log(-s-1) (-s-1)^{H-\frac{1}{2}} \\
        &\quad+ \mathds{1}_{s<2} \log(-s-2) (-s-2)^{H-\frac{1}{2}}, \\
        \leadsto\quad \left|\frac{d}{dH} \bar{g}(s,H)  \right| 
        &\leq C (1+|\log s|) \min\left(|s|^{H-\frac{1}{2}}, |s|^{H-\frac{5}{2}}\right).
    \end{align*}
    Hence, the mean value theorem yields some $\tilde{H}_s$ between $H_{\frac{s}{n}}$ and $H_{\frac{s-1}{n}}$ such that, for $k\geq 1$,
    \begin{align*}
        & \quad \int_{i-k-1}^{i-k}  |\bar{g}(s-i, H_{\frac{s}{n}}) - \bar{g}(s-i, H_{\frac{s-1}{n}})|^2\, ds \\
        &\leq  \int_{i-k-1}^{i-k} |H_{\frac{s}{n}} - H_{\frac{s-1}{n}}|^2 \left|\frac{d}{dH} \bar{g}_{0,n}(s-i,H)|_{H=\tilde{H}_s}  \right|^2\, ds \\
        &\leq C n^{-2\eta} \int_{i-k-1}^{i-k} [1+|\log (s-i)|] \left[\mathds{1}_{|s-i|<1} |s-i|^{2\tilde{H}_{s}-1} + \mathds{1}_{|s-i|\geq 1} |s-i|^{2\tilde{H}_{s}-5} \right]\, ds \\
        &\leq C n^{-2\eta} \int_{i-k-1}^{i-k} [1+|\log (s-i)|] |s-i|^{2\tilde{H}_{s}-5} \, ds \\
        &\leq C n^{-2\eta} (k+1)^{2H_{\frac{i-k}{n}}-3+Ck^{-\eta}} \log(k)\\
        &\leq C n^{-2\eta} (k+1)^{2H_{\frac{i-k}{n}}-3+\delta},
    \end{align*}
    for $\delta>0$ small.
    This term may be further bounded as in \eqref{eqn:phys-dep-mbm} to find
    \begin{align*}
        \log(n) n^{-\eta} n^{2H_{\frac{i-1}{n}} - 2 H_{\frac{i-k}{n}}} \int_{i-k-1}^{i-k}  |\bar{g}(s-i, H_{\frac{s}{n}}) - \bar{g}(s-i, H_{\frac{s-1}{n}})|^2\, ds 
        &\leq C (k+1)^{-3}.
    \end{align*}
    Using this in \eqref{eqn:intpar-TV-1}, and exploiting the Gaussianity, we obtain for all $q\geq 2$
    \begin{align*}
        \left\|\widetilde{G}_{i,n}(\beps_0) - \widetilde{G}_{i-1,n}(\beps_0)\right\|_{L_q}^2
        &\leq C n^{-2\eta} \sum_{k=0}^\infty (k+1)^{-3} \leq C n^{-2\eta}.
    \end{align*}
    Since $\|\tilde{G}_{i,n}\|_{L_q}$ is bounded, it is straightforward to conclude from this and \eqref{eq:defG} that
    \begin{align}
        \left\|G_{i,n}(\beps_0) - G_{i-1,n}(\beps_0)\right\|_{L_2}^2
        &\leq C n^{-2\eta}, \label{eqn:kernel-regularity}
    \end{align}
    and hence \eqref{ass:A1} holds with $\Theta_n=\mathcal{O}(1)$ and $\Gamma_n = n^{1-\eta}$.

    \noindent\underline{Ad \eqref{ass:A3}:}
    For $u\in[0,1]$, define the limiting kernels
    \begin{align}
        \widetilde{G}_u(\beps_j) &= \int_{-\infty}^0 \bar{g}(s, H_u)\, d\tilde{B}_{s+j},\nonumber
    \intertext{and, in view of \eqref{eq:defG},} 
        G_u(\beps_j) &= \begin{pmatrix}
            \widetilde{G}_u(\beps_j)^2 \\ \left(\widetilde{G}_u(\beps_j) + 2\widetilde{G}_u(\beps_{j-1}) + \widetilde{G}_u(\beps_{j-2})\right)^2.
        \end{pmatrix} \label{eqn:kernel-limit}
    \end{align}
    Then
    \begin{align*}
        &\quad \| \widetilde{G}_{u}(\beps_0) - \widetilde{G}_{\lfloor un\rfloor, n}(\beps_0)\|_{L_q} \\
        &\leq C \|\widetilde{G}_{u}(\beps_0) - \widetilde{G}_{\lfloor un\rfloor, n}(\beps_0) \|_{L_2} \\
        &\leq \int_{-\infty}^{\lfloor un\rfloor} |\bar{g}(s-\lfloor un\rfloor, H_u) - \tilde{g}_{\lfloor un\rfloor, n}(s)|^2\, ds \\
        &= \int_{-\infty}^{0} \left|\bar{g}(s, H_u) - \frac{\sigma_{\frac{\lfloor un\rfloor +s}{n}}}{\sigma_{\frac{\lfloor un\rfloor}{n}}} n^{H_{\frac{\lfloor un\rfloor}{n}} - H_{\frac{\lfloor un\rfloor+s}{n}}  } \bar{g}(s, H_\frac{\lfloor un\rfloor+s}{n})\right|^2\, ds \\
        &\leq 2\int_{-\infty}^0 \left|\bar{g}(s, H_u) - \bar{g}(s, H_{\frac{\lfloor un\rfloor+s}{n}}) \right|^2\, ds \\
        &\quad + 2\int_{-\infty}^0 \left|1-\frac{\sigma_{\frac{\lfloor un\rfloor +s}{n}}}{\sigma_{\frac{\lfloor un\rfloor}{n}}} n^{H_{\frac{\lfloor un\rfloor}{n}} - H_{\frac{\lfloor un\rfloor+s}{n}}  }\right|^2 \left| \bar{g}(s, H_\frac{\lfloor un\rfloor+s}{n})\right|^2\, ds \\
        &\leq C n^{-2\eta} \int_{-1}^0 \,ds + C \int_{-\infty}^{-1}  (\tfrac{|s|}{n})^{2\eta} \sup_{H\in [\underline{H}, \overline{H}]} \left| \tfrac{d}{dH} \overline{g}(s, H) \right|^2 \, ds \\
        &\quad  + C \log(n)  \int_{-n^{1-\delta}}^{-1} (\tfrac{|s|}{n})^{2\eta}  \sup_{H\in [\underline{H}, \overline{H}]} \left| \overline{g}(s, H) \right|^2 \, ds \\
        &\quad  + C \log(n)  \int_{-\infty}^{-n^{1-\delta}} n^{2\overline{H}} \sup_{H\in [\underline{H}, \overline{H}]} \left| \overline{g}(s, H) \right|^2 \, ds \\
        &\leq C n^{-2\eta} + C\log(n) n^{-2\eta} \int_{-\infty}^{-1} |s|^{2\overline{H}+2\eta-5}\, ds + C\log(n)\int_{-\infty}^{-n^{1-\delta}} |s|^{2\overline{H}-5}\, ds\\
        &\leq C \log(n)n^{-2\eta} + C \log(n) n^{-2(1-\delta)} \quad \leq C \log(n) n^{-2\eta}
    \end{align*}
    for any small $\delta>0$. 
    Here, we used that $|\bar{g}(s, H)| \leq C|s|^{H-\frac{5}{2}}$ and $\frac{d}{dH}|\bar{g}(s,H)|\leq C (1+\log|s|)|s|^{H-\frac{5}{2}} $ for $|s|\geq 1$.
    That is, we find that for any $q\geq 2$
    \begin{align*}
        \|\widetilde{G}_{\lfloor un\rfloor, n}(\beps_0) - \widetilde{G}_{u}(\beps_0)\|_{L_q}
        \leq C n^{-\eta}.
    \end{align*}
    Again, this directly yields that
    \begin{align}
        \|G_{\lfloor un\rfloor, n}(\beps_0) - G_{u}(\beps_0)\|_{L_q}
        \leq C n^{-\eta}, \label{eqn:kernel-limit-conv}
    \end{align}
    and hence $\int_0^1  \|G_{\lfloor un\rfloor, n}(\beps_0) - G_{u}(\beps_0)\|_{L_2} \, du \to 0$ as $n\to\infty$. 
    This establishes \eqref{ass:A3}.

    \noindent\underline{Rate constraints:}
    In the previous steps, we have verified the conditions of Theorem \ref{thm:multiplier-FCLT} for any $q\geq 2$, for $\beta=\frac{3}{2}$, and with $\Gamma_n\asymp n^{1-\eta}$, $\Lambda_n^2\asymp \log(n)^2n^{\frac{1}{2}-r} + L_n^{2\eta} n^{1-2\eta}$, $\Psi_n\asymp n^{1-\eta}$, $\Phi_n=\mathcal{O}(1)$, $\Theta_n=\mathcal{O}(1)$. 
    Upon choosing $q$ large enough, we have $\xi(q,\beta) = \frac{\beta-1}{4\beta-2} = \frac{1}{8}$, hence
    \begin{align*}
        (\Gamma_n + \Psi_n)^{\frac{\beta-1}{2\beta}} \sqrt{\log(n)} n^{-\xi(q,\beta)} = n^{\frac{1-\eta}{6}} \sqrt{\log(n)} n^{-\frac{1}{8}},
    \end{align*}
    which tends to zero if $\eta > \frac{1}{2}$.
    Moreover, upon choosing $q$ large enough, we have
    \begin{align*}
        n^{-\frac{1}{2}} \Lambda_n^2 + \Lambda_n L_n^{1-\beta} + n^{\frac{1}{q}-\frac{1}{2}}L_n \to 0 \\
        \Longleftarrow n^{\frac{1}{2} -r} \log(n)^2 \ll L_n \ll n^{\frac{1}{2}-\delta} \quad \text{for some $\delta>0$, and } L_n \ll n^{1-\frac{1}{4\eta}}.
    \end{align*}
    Thus, we have shown that Theorem \ref{thm:multiplier-FCLT} is applicable under the conditions formulated in Theorem \ref{thm:integrated-clt}.

    \noindent\underline{Determining the asymptotic variance:}
    The asymptotic local variance is given by
    \begin{align*}
        \Sigma_u &= \sum_{h=-\infty}^\infty g_u \Cov(G_u(\beps_0), G_u(\beps_h)) g_u^T, \\
        g_u &= \frac{1}{2\,\Gamma_{H_u}(0)}\begin{pmatrix}
            -1\\ 2^{-2H_u}
        \end{pmatrix}.
    \end{align*}
    The autocovariances may be computed as in Lemma \ref{lem:cov}, hence $\Cov(G_u(\beps_0), G_u(\beps_h)) = \Sigma_{H_u}(h)$ as therein. 
    Thus,
    \begin{align}
        \Sigma_u &= \tau^2(H_u), \nonumber \\
        \tau^2(H)&=\frac{1}{4 \Gamma_{H}(0)^2} \sum_{h=-\infty}^\infty (-1, 2^{-2H}) \Sigma_{H}(h) \begin{pmatrix}
            -1 \\ 2^{-2H}
        \end{pmatrix} \nonumber \\
        &= \frac{1}{2 \Gamma_{H}(0)^2} \sum_{h=-\infty}^\infty \Big\{ \Gamma_{H}(h)^2 + 2^{-4H} (2\Gamma_{H}(h) + \Gamma_{H}(h-1)+\Gamma_{H}(h+1))^2 \nonumber \\
        &\qquad \qquad - 2^{-2H} \left[ (\Gamma_{H}(h) + \Gamma_H(h-1))^2 + (\Gamma_H(h)+\Gamma_H(h-1))^2  \right] \Big\} \nonumber \\
        \begin{split}
        &= \frac{1}{2 \Gamma_{H}(0)^2} \sum_{h=-\infty}^\infty \Big\{ \Gamma_{H}(h)^2 + 2^{-4H} (2\Gamma_{H}(h) + \Gamma_{H}(h-1)+\Gamma_{H}(h+1))^2 \\
        &\qquad \qquad - 2^{-2H+1}  (\Gamma_{H}(h) + \Gamma_H(h-1))^2 \Big\}.
        \end{split} \label{eqn:asymp-var}
    \end{align}
\end{proof}


\begin{proof}[Proof of Theorem \ref{thm:studentizing}]
    By Theorem \ref{thm:rate}, and as established in the proof of Theorem \ref{thm:integrated-clt}, we have for any $q\geq 2$
    \begin{align*}
        \max_{t=2L_n,\ldots, n} |\widehat{H}_n(\tfrac{t}{n}) - H_{\frac{t}{n}}| \quad\pconv\quad 0.
    \end{align*}
    Hence,
    \begin{align*}
        \sup_{u\in[0,1]} \left|\frac{1}{n}\sum_{t=2L}^{\lfloor un\rfloor} \tau^2(H(\tfrac{t}{n})) - \frac{1}{n}\sum_{t=2L}^{\lfloor un\rfloor} \tau^2(\widehat{H}_n(\tfrac{t}{n})) \right| \pconv 0.
    \end{align*}
    Moreover, continuity of $v\mapsto H_v$, and $L_n/n\to 0$, yields
    \begin{align*}
        \sup_{u\in[0,1]} \left|\int_0^u \tau^2(H_v) \, dv - \frac{1}{n}\sum_{t=2L_n}^{\lfloor un\rfloor} \tau^2(H(\tfrac{t}{n})) \right| \to 0.
    \end{align*}
\end{proof}


\begin{proof}[Proof of Proposition \ref{prop:CUSUM}]
    Convergence under the null is a direct consequence of Theorems \ref{thm:integrated-clt} and \ref{thm:studentizing}.
    Under the alternative, the function $u\mapsto \mathcal{H}(u)$ is not linear.
    This holds because $v\mapsto H_v$ is continuous, and hence $H_v$ can not deviate at a single point $v$. 
    Thus, $\sup_{u\in[0,1]} |\mathcal{H}(u)-u \mathcal{H}(1)|>0$.
    Moreover, Theorem \ref{thm:integrated-clt} yields $\|\widehat{\mathcal{H}}- \mathcal{H}\|_\infty = \mathcal{O}(1/\sqrt{n})$ and thus $\sqrt{n} T_{\mathrm{CUSUM}}(\widehat{\mathcal{H}}) \to \infty$ in probability, at rate $\sqrt{n}$.
    On the other hand, $q_n(\alpha)\pconv q(\alpha)$ as a consequence of Theorem \ref{thm:studentizing}.
    This establishes consistency of the test.
\end{proof}

\begin{proof}[Proof of Proposition \ref{prop:GOF}]
    The first claim holds because $q_n(\alpha)\to q(\alpha)$ by Theorem \ref{thm:studentizing}, where $q(\alpha)$ is the $1-\alpha$ quantile of $Z=\sup_u |W(\int_0^u \tau^2(H_v)\, dv)|$, and $\sqrt{n}\widehat{T}(\mathcal{G}_0) \leq \sqrt{n}\widehat{T}(H)\wconv Z $.
    
    We proceed to prove the second claim.
    If $\mathcal{G}_0$ is closed, then $\widetilde{\mathcal{G}}_0 = \{ u\mapsto \int_0^u \tilde{H}(v)\, dv\,|\, \tilde{H}\in\mathcal{G}_0 \}$ is also closed, and $\int H(v)\, dv\notin \widetilde{\mathcal{G}}_0$. Hence, $\sup_{u\in[0,1]} |\int_0^u \tilde{H}(v) \, dv - \int_0^u \tilde{H}(v) \, dv| > 0$ for all $\tilde{H}\in \mathcal{G}_0$. 
    As $\sqrt{n}(\widehat{\mathcal{H}} - \mathcal{H}) = o_P(1)$, this implies that $\sqrt{n}\widehat{T}(\mathcal{G}_0)\pconv \infty$.
    On the other hand, Theorem \ref{thm:studentizing} yields convergence of $q_n(\alpha)$ to a finite number, establishing consistency of the test.
\end{proof}

\bibliography{itombm}
\bibliographystyle{apalike}

\end{document}